\documentclass[11pt,a4paper,reqno]{amsart}
    \usepackage{vmargin,color}
    \usepackage[utf8]{inputenc}
    \usepackage[colorlinks=true,linkcolor=black,citecolor=black]{hyperref}
    \usepackage{amsmath,amsfonts,amsthm,epsfig,graphicx,amssymb}
    \usepackage{mathtools}
    \usepackage{esint}
    \usepackage{caption}

    \newcommand{\CC}{\mathbf{C}}
    \newcommand{\DD}{\mathbf{D}}
    \newcommand{\cald}{\mathcal{D}}

    \newcommand{\E}{\mathcal{E}}
    \newcommand{\F}{\mathcal F}
    \newcommand{\G}{\mathcal G}
    
    \renewcommand{\H}{\mathcal{H}}

    \renewcommand{\P}{\mathcal{P}}
    \newcommand{\Sc}{\mathcal S}
    
    \newcommand{\X}{\mathcal{X}}

    \newcommand{\N}{\mathbb{N}}

    \newcommand{\R}{\mathbb{R}}
    \renewcommand{\SS}{\mathbb{S}}

    \newcommand{\NN}{\mathcal{N}}
    \newcommand{\La}{\Lambda}
    
    \renewcommand{\S}{\Sigma}
    \renewcommand{\a}{\alpha}
    \renewcommand{\b}{\beta}
    
    \newcommand{\de}{\delta}
    \newcommand{\e}{\varepsilon}

    \newcommand{\s}{\sigma}
    \renewcommand{\t}{\theta}
    \newcommand{\om}{\omega}
    
    \newcommand{\Lip}{{\rm Lip}}

    \newcommand{\Div}{{\rm div}\,}
    
    \newcommand{\dist}{{\rm dist}}

    \newcommand{\weakstar}{\stackrel{\scriptscriptstyle{*}}{\rightharpoonup}}
    
    \newcommand{\ov}{\overline}
    
    \newcommand{\pa}{\partial}

    \newcommand{\eps}{\e}

    \newcommand{\h}{\H^n}
    \newcommand{\rn}{\mathbb{R}^{n+1}}

    \newcommand{\A}{\mathcal{A}}
    \newcommand\restr[2]{{
      \left.\kern-\nulldelimiterspace 
      #1 
      \right|_{#2} 
      }}
    
    \newcommand{\lip}{\mathrm{Lip}}

    \theoremstyle{plain}
    \newtheorem{theorem}{Theorem}[section]
    \newtheorem{lemma}[theorem]{Lemma}

    \newtheorem*{theorem*}{Theorem}
    \newtheorem*{corollary*}{Corollary}

    \theoremstyle{definition}

    \newtheorem{remark}[theorem]{Remark}

    \newtheorem*{notation*}{Notation}

    \numberwithin{equation}{section}
    \numberwithin{figure}{section}

    \usepackage{enumitem}
    \newcommand{\hnsub}{\stackrel{\h}{\subset}}
    \newcommand{\pas}{\pa^*}
    
    \newcommand{\hneq}{\stackrel{\h}{=}}
    \newcommand{\ps}{P^*_{\rm wet}}
    
    \newcommand{\p}{\rho}
    \newcommand{\MS}{\text{MS}}

    \allowdisplaybreaks

\title{Isoperimetric clusters with a liquid phase}

\author{Jake Wellington}
\address{Department of Mathematics, The University of Texas at Austin, 2515 Speedway, Stop C1200, Austin TX 78712-1202, United States of America}
\email{wellingtj@utexas.edu}

\begin{document}

\begin{abstract}
    A wet foam is a cluster of air chambers whose interfaces are coated by a single liquid chamber of prescribed volume. In the corresponding isoperimetric problem the liquid chamber may collapse onto films of zero thickness, so that minimizers need not exist. We identify the relaxed energy accounting for such collapsed films and prove that limits of minimizing sequences are its minimizers. The resulting compactness and energy representation theorems lead to the regularity of these generalized minimizers, including an equilibrium law for the pressures across collapsed films. Finally, as the liquid volume vanishes, the minimal energies converge to those of the classical dry foam.
\end{abstract}


\maketitle

\tableofcontents

\section{Introduction}

\subsection{Overview} A foam can most simply be described as some number of bubbles packed together. Almgren's theory of isoperimetric clusters as outlined in Part IV of \cite{Maggi_Book} provides the prevailing mathematical theory behind what is referred to as a `dry' foam (see \cite{weairefoam} Section 1.3 or \cite{cantatfoam} Chapter 2, Sections 3 and 4). An $N$-cluster is a family $\E=\{\E(h)\}_{h=1}^N$ of $N$-many sets of finite perimeter in $\R^{n+1}$
with {\bf exterior} given by $$ \E(0) = \rn\setminus\bigcup_{h=1}^N\E(h).$$ Whenever we modify the chambers $\E(1),\dots,\E(N)$ of a cluster, the exterior chamber is always understood to be redefined as the complement of their union. The {\bf volume of $\E$} is the vector
\[
|\E|=\big(|\E(1)|,\cdots,|\E(N)|\big)\,,\qquad 0<|\E(h)|<\infty\,.
\]

Using a soap bubble as a concrete example, the air inside the bubbles is given by the chambers $\E(h)$, whereas the soap itself will be the boundaries $\pa\E(h)$. A dry foam which can appear `in nature' will be the one that encloses the correct volume of air, while minimizing some notion of perimeter. While simpler to discuss mathematically, a dry foam has a few major disadvantages in modeling capacity that are discussed in \cite{King_2022}. The first of these issues is that the soap itself is assumed to have no volume, which introduces a problem with scaling. Given a `natural' foam configuration, scaling the entire foam by twofold or even a thousandfold will still produce a `natural' foam configuration which minimizes perimeter given certain volume constraints. In other words, there is no size restriction for a bubble formed by a vanishingly small amount of fluid in the dry model. 

This issue is rectified by introducing the theory of a `wet' foam, which amounts to giving volume to the liquid. It is easiest to understand what this means in terms of a dry foam. Given a dry foam with $N$ chambers (regions of air), a wet foam would look like a dry foam with $N+1$ chambers, where the extra chamber corresponds to the liquid, and should engulf the other chambers. Rigorously, given $\e>0$, a {\bf wet $N$-cluster} with liquid volume $\e$, is an $(N+1)$-cluster $\E$ with
\[
|\E(N+1)|=\e\,,
\]
and such that
\[
\bigcup_{h=1}^N\pa^*\E(h)\stackrel{\H^n}{\subset}\pa^*\E(N+1)\,.
\] 
Perhaps the easiest case to illustrate is the case of the two dimensional 1-cluster. The dry model is famously a circle. The wet model is going to be two concentric circles. The {\bf perimeter} of a wet $N$-cluster is defined by
\begin{eqnarray*}
P_{\rm wet}(\E)\!\!&=&\!\!P(\E(N+1))\,.
\end{eqnarray*}
We are interested in the {\bf wet isoperimetric problem}
\begin{equation}\label{main problem}
    \psi_\e(v_1,...,v_N)=\inf\Big\{P_{\rm wet}(\E):|\E| =(v_1,...,v_N,\e)\Big\}\,,
\end{equation}
where the infimum ranges among wet $N$-clusters. Minimizing sequences may develop collapsed liquid films, and in general the infimum need not be attained in the class of wet clusters: for $N=1$ it is (by concentric balls), while for $N\geq2$ and small $\e$ one expects that it is not. We do not pursue this question here.

This paper is going to be dedicated entirely to fleshing out the basic theory behind the wet isoperimetric problem. This is very much in the flavor of \cite{King_2022}, where this theory is developed in the case of the Plateau problem.

\subsection{Main results} 
 Minimizing sequences converge to $(N+1)$-clusters $\E$ that satisfy $|\E|=(v_1,...,v_N,\e)$ and minimize the following {\bf relaxed wet perimeter} functional
\[
\ps(\E)=P(\E(N+1))+2\,\sum_{0\le i<j\le N}\H^n\big(\pas\E(i)\cap\pas\E(j)\big)\,.
\]
Sometimes, it is more convenient to work with the equivalent form
\begin{equation}\label{wet energy equivalent form}
    \ps(\E) = \sum_{h=0}^{N}P(\E(h)).
\end{equation}

\begin{theorem}\label{existence thm main}
  For every $\e$, $v_1$, ..., $v_N$ positive numbers and $N\ge 1$, there exists an $(N+1)$-cluster $\E$ such that
  \begin{eqnarray}
  \label{gen min 1}
  |\E|\!\!&=&\!\!\big(v_1,...,v_N,\e\big)\,,
  \\
  \label{gen min 2}
  P^*_{\rm wet}(\E)\!\!&=&\!\!\psi_\e(v_1,...,v_N)\,.
  \end{eqnarray}
\end{theorem}\noindent
This motivates us to consider the \textbf{relaxed wet isoperimetric problem}
\begin{equation}\label{relaxed problem defn}
    \varphi_\e(v_1,\dots,v_N) = \inf\Big\{P^*_{\rm wet}(\E):|\E|_{\rm liquid} = \e, |\E|_{\rm air} = (v_1,\dots,v_N)\Big\},
\end{equation}
where the infimum now ranges over all $(N+1)$-clusters. As it turns out, the generalized minimizers of the wet isoperimetric problem are regular minimizers for the relaxation. This is nicely summarized by the following theorem:
\begin{theorem}\label{relaxation thm main} If $v_1,\dots, v_N > 0$, $\psi_\eps$ is as in \eqref{main problem} and $\varphi_\eps$ is as in \eqref{relaxed problem defn}, then for any $\eps > 0$
    \begin{equation*}
        \varphi_\e(v_1,\dots,v_N) = \psi_\e(v_1,\dots,v_N).
    \end{equation*}
\end{theorem}
This theorem leads to the conclusion that if $\E$ satisfies \eqref{gen min 1} and \eqref{gen min 2}, then the air chambers of $\E$ are $(\Lambda,r)$-perimeter minimizers. We use regularity theory associated with $(\La,r)$-minimizers as the primary tool to deduce $C^{1,\alpha}$ regularity of the chamber interfaces for minimal clusters away from a small dimensional singular set, and potential cusps formed by the singular part of the wet chamber. We remark this regularity is often defined in terms of disks and cylinders given by
    \begin{align*}
        \DD^\nu_r(x) &= x + \{y\in\nu^\perp:|y|<r\},\\
        \CC_r^\nu(x) &= \{y + t\nu : y\in \DD^\nu_r(x), |t|<r\}.
    \end{align*}
    Moreover, $C^{1,\a}$ regularity allows us to invoke a sharp regularity result for the two membranes problem. 
\begin{theorem}\label{regularity thm main}
    If an $(N+1)$-cluster $\E$ satisfies \eqref{gen min 1} and \eqref{gen min 2}, then, setting
    \begin{equation}\label{Singular set defn}
        K = \bigcup_{h=0}^{N+1}\pa\E(h),\,\,\,\Sigma = K\setminus\bigcup_{h=0}^{N+1} \pas\E(h),\,\,\,\S' = \pa\E(N+1)\setminus\pas\E(N+1)\,,
    \end{equation}
    the following hold:
  
    \medskip
  
    \noindent {\bf (i):} $\S = \emptyset$ for $1\leq n\leq 6$, $\S$ is locally finite for $n=7$ and $\H^s(\S) = 0$ for $s>n-7$ for $n\geq 8$;
    
    \medskip
  
    \noindent {\bf (ii):} $K\setminus\big(\S\cup \S'\big)$ is a smooth hypersurface, and there are constants $\lambda_1,\dots,\lambda_{N+1}\in\R$ (we set $\lambda_0 = 0$) such that, for every $i\neq j$, the interface $\pas\E(i)\cap\pas\E(j)\setminus(\S\cup\S')$ has constant mean curvature $(\lambda_i-\lambda_j)/(1+\mathbf{1}\{i,j\neq N+1\})$ with respect to $\nu_{\E(i)}$ (with the convention that a sphere has positive mean curvature with respect to its outer normal);

    \medskip

    \noindent {\bf (iii):} For every $x\in \S'\setminus\S, \,K$ is the union of two $C^{1,1}$ hypersurfaces which detach tangentially. More precisely, there are $r>0, \nu\in\SS^{n}$ and $u_1,u_2\in C^{1,1}(\DD^\nu_r(x))$ such that
    \begin{align*}
        u_1(x) = u_2(x) = 0, \,\,\,\,\,\,\,\,\,\,\,\, u_1\leq u_2 \text{ on } \DD^\nu_r(x),\,\,\,\,\,\,\,\,\,\,\,\, \{u_1<u_2\} \text{ is nonempty}
    \end{align*}
    and
    \begin{align}\label{c11 cap K}
        \CC^\nu_r(x) \cap K &= \bigcup_{i=1,2}\Big\{y+u_i(y)\nu:y\in \DD^{\nu}_r(x)\Big\}\\\label{c11 bdry}
        \CC^\nu_r(x) \cap \pas\E(N+1) &= \bigcup_{i=1,2}\Big\{y+u_i(y)\nu:y\in\{u_1<u_2\}\Big\}\\\label{c11 chamber}
        \CC^\nu_r(x) \cap \E(N+1) &= \bigcup_{i=1,2}\Big\{y+t\nu:y\in \DD^{\nu}_r(x),t\in(u_1(y),u_2(y))\Big\}.
    \end{align}
    Moreover, if $\E(i)$ and $\E(j)$ are the two chambers adjacent to $\E(N+1)$ in $\CC^\nu_r(x)$, then
    \begin{equation}\label{averaged pressure inequality}
        \lambda_i+\lambda_j > 2\lambda_{N+1}\,,
    \end{equation}
    that is, the mean curvatures $H_1$, $H_2$ of the graphs of $u_1$, $u_2$ with respect to $\nu$ satisfy $H_1 - H_2 = \lambda_i+\lambda_j-2\lambda_{N+1}>0$ on $\{u_1<u_2\}$: the two sheets detach with strictly positive relative curvature. In particular, at every point of $\S'\setminus\S$ at least one of the two adjacent chambers is at strictly higher pressure than the liquid and two chambers with $\lambda_h\leq\lambda_{N+1}$ are never separated by a collapsed film.

    \medskip

    \noindent {\bf (iv):} We have $\Gamma = \S'\setminus\S = \Gamma_{r}\cup\Gamma_s, \Gamma_{r}\cap\Gamma_s = \emptyset$. For every $x\in\Gamma_r$ there are $\rho > 0$ and $\beta\in(0,1)$ such that $\Gamma_r\cap B_\rho(x)$ is a $C^{1,\b}$ embedded $(n-1)$-dimensional manifold. $\Gamma_s$ is relatively closed in $\Gamma$ and can be partitioned into a family $\{\Gamma_s^k\}_{k=0}^{n-1}$ where $\Gamma^k_s$ is locally $\H^k$-rectifiable. 
\end{theorem}

Finally, we show that the relaxed problem converges to the classical (dry) isoperimetric cluster problem as the liquid volume vanishes. Recall that the perimeter of a dry $N$-cluster $\F$ is
\begin{equation}\label{dry perimeter}
    P(\F) = \sum_{0\le h<k\le N}\h\big(\pas\F(h)\cap\pas\F(k)\big) = \frac{1}{2}\sum_{h=0}^N P(\F(h)),
\end{equation}
see \cite[Section 29]{Maggi_Book}. We define the \textbf{dry isoperimetric energy}
\begin{equation}\label{dry problem defn}
    \varphi_0(v_1,\dots,v_N) = 2\inf\Big\{P(\F) : |\F| = (v_1,\dots,v_N)\Big\},
\end{equation}
where the infimum ranges over $N$-clusters. By \eqref{dry perimeter}, $\varphi_0(v) = \inf\{\sum_{h=0}^NP(\F(h)):|\F| = v\}$ is the formal $\e = 0$ case of \eqref{relaxed problem defn}, and the infimum is attained by \cite[Theorem 29.1]{Maggi_Book}.
\begin{theorem}\label{convergence to plateau}
    If $v_1,\dots,v_N>0$, $\varphi_\e$ is as in \eqref{relaxed problem defn} and $\varphi_0$ is as in \eqref{dry problem defn}, then
    $$\lim_{\varepsilon \to 0^+} \varphi_\e(v_1,\dots,v_N) = \varphi_0(v_1,\dots,v_N).$$
\end{theorem}

\addtocontents{toc}{\protect\setcounter{tocdepth}{-1}}
\section*{Acknowledgements}
The author would like to thank Francesco Maggi for his invaluable guidance on this project, and Kenneth DeMason for many helpful discussions.
This work was supported by NSF Grant DMS-2247544 and by NSF RTG Grant DMS-1840314.

\section*{AI Statement}
The author used an AI assistant (Anthropic's Claude) to check the manuscript for errors and to verify references; all mathematical content is the author's own.
\addtocontents{toc}{\protect\setcounter{tocdepth}{1}}


\section{Cup, cone and slab competitors}\label{sec:competitors} 
We divide this section into several subsections. The first is dedicated to translating the compactness results for clusters from Chapter 29 of \cite{Maggi_Book} into the wet setting. The second is dedicated to proving some basic lemmas about clusters and sets of finite perimeter which will be useful in the process of constructing carefully chosen competitors to a given minimizing sequence. These competitors are the cup competitors, cone competitors and slab competitors seen in \cite{King_2022}. The final three are dedicated to ensuring these competitors are well defined in the setting of wet isoperimetric clusters, and demonstrating why they are useful in this setting. 

\subsection{Compactness Tools.}
In order to apply compactness results to a sequence of wet clusters, we require an adapted version of the truncation lemma from \cite{Maggi_Book} Section 29.4, which looks as follows: Here and below, $d(\E,\F) = \sum_{h=1}^{N+1}|\E(h)\Delta\F(h)|$ denotes the $L^1$-distance between two $(N+1)$-clusters, and, for a set $A$ contained in a $C^1$ hypersurface $M$ with a chosen unit normal $\nu_M$, $N_\de(A) = \{y + t\nu_M(y) : y\in A,\ 0<t<\de\}$ denotes the one-sided normal neighborhood of $A$ of width $\de$ (in Remark \ref{wet truncated cluster remark} we take $M = \{u = r_0\}$ with $\nu_M = \nabla u$, and in Section \ref{subsec:cup} we take $M = \pa B_r(x)$ with $\nu_M$ the inner normal).
\begin{lemma}\label{truncation lemma}
    Let $F$ be a closed set, $u(x) = {\rm dist}(F,x)$, and $\E$ a wet $N$-cluster satisfying 
    \begin{equation}
        \sum_{h=1}^{N+1} \big|\E(h)\setminus F\big| \leq \a.
    \end{equation}
    If we define the set
    \begin{equation*}
        H = \Big\{r\in[0,11(n+1)\a^\frac{1}{n+1}] :\H^{n-1}\big(\bigcup_{h=1}^{N+1}\pas\E(h)\cap \{u =r\}\big)<\infty\Big\},
    \end{equation*}
    then there exists $r_0\in H$ such that the $(N+1)$-cluster $\E'$ defined by 
    \begin{align*}
        \E'(h) = \E(h) \cap\{u\leq r_0\}, && 1\leq h\leq N+1
    \end{align*} satisfies
    \begin{equation}\label{truncation estimates}
        P_{\rm wet}(\E') + 2\sum_{h=1}^N \h\big(\E(h)\cap\{u=r_0\}\big) \leq P_{\rm wet}(\E) - \frac{d(\E,\E')}{4\alpha^{\frac{1}{n+1}}}.
    \end{equation}
\end{lemma}
\begin{proof}
    Set $c_\a = 11(n+1)\a^{\frac{1}{n+1}}$. In the event that $\E(h)\subset\{u\le c_\a\}$ for every $h$, we set $r_0 = c_\a$. Otherwise, consider the function 
    \begin{equation*}
        m(r) = \sum_{h=1}^{N+1} |\E(h) \cap\{u>r\}|
    \end{equation*}
    and the truncated cluster $$\E^r(h) = \E(h)\cap\{u\le r\}.$$
    The coarea formula will imply that
    \begin{equation*}
        m(r) = \sum_{h=1}^{N+1} \int_r^\infty \h\big(\E(h)\cap\{u=s\}\big)ds,
    \end{equation*}
    so for almost every $r$ it holds that
    \begin{equation*}
        m'(r) = -\sum_{h=1}^{N+1}\h\big(\E(h)\cap\{u=r\}\big).
    \end{equation*}
    As it holds that $H\subset$ supp$(m)$, suppose that for every $r\in H$ it holds that 
    \begin{equation}\label{truncation contradiction first step}
        P\big(\E(N+1)\big) < P\big(\E^r(N+1\big)) + 2\sum_{h=1}^N \h\big(\E(h)\cap\{u=r\}\big)+ \frac{m(r)}{4\alpha^{\frac{1}{n+1}}}
    \end{equation} where $\E^r = \E \cap\{u\le r\}.$
    As $\E\cap\{u<r\} = \E^r\cap\{u<r\}$, \eqref{truncation contradiction first step} will imply that if $r\in H$,
    \begin{equation}\label{truncation contradiction second step}
         P\big(\E(N+1);\{u>r\}\big) < -m'(r) + \sum_{h=1}^N \h\big(\E(h)\cap\{u=r\}\big)+ \frac{m(r)}{4\alpha^{\frac{1}{n+1}}}.
    \end{equation}
    As $\E$ is a wet $N$-cluster, observing that $m(r) \leq \alpha^{\frac{1}{n+1}}m(r)^{\frac{n}{n+1}}$ and \eqref{truncation contradiction second step} will further imply that
    \begin{equation}\label{truncation contradiction third step}
        \sum_{h=1}^{N+1}\frac{P(\E(h);\{u>r\})}{2} < -2m'(r) + \frac{m(r)^{\frac{n}{n+1}}}{4}.
    \end{equation}
    Next, we use the fact $P(\E(h);\{u>r\}) = P(\E(h)\cap\{u>r\}) - \h(\E(h)\cap\{u=r\})$ to see that \eqref{truncation contradiction third step} implies that
    \begin{equation}\label{truncation contradiction fourth step}
        \sum_{h=1}^{N+1}\frac{P(\E(h)\cap\{u>r\})}{2} < -\frac{5}{2}m'(r) + \frac{m(r)^{\frac{n}{n+1}}}{4}.
    \end{equation}
    Applying the non-sharp isoperimetric inequality to \eqref{truncation contradiction fourth step} will yield
    \begin{equation*}
        \frac{m(r)^\frac{n}{n+1}}{2} < -\frac{5}{2}m'(r) + \frac{m(r)^\frac{n}{n+1}}{4}
    \end{equation*} or equivalently,
   \begin{equation}\label{truncation contradiction final step}
       (m(r)^\frac{1}{n+1})' < -\frac{1}{10(n+1)}.
   \end{equation}
   Integrating \eqref{truncation contradiction final step} over $H$ will yield
   \begin{equation*}
       m(c_{\a})^{\frac{1}{n+1}} - m(0)^\frac{1}{n+1} < -\frac{c_\a}{10(n+1)} = -\frac{11}{10}\a^{\frac{1}{n+1}}.
   \end{equation*}
   As $m\geq 0$, we have that 
   \begin{equation*}
       m(0)^{\frac{1}{n+1}} > \frac{11}{10}\a^{\frac{1}{n+1}} > \alpha^\frac{1}{n+1} \ge m(0)^\frac{1}{n+1},
   \end{equation*}
    a contradiction.
\end{proof}
\begin{remark}\label{wet truncated cluster remark}
    We can further `cap off' the exposed regions from truncation as in the case of the cup competitor. Indeed, we can define the cluster
    \begin{align*}
        \F(h) &= \E^{r_0}(h) &{\rm for}\,\, 0\le h \le N\\
        \F(N+1) &= \E^{r_0}(N+1) \cup N_{\delta}\Big(\bigcup_{h=1}^N \E(h) \cap\{u = r_0\}\Big)
    \end{align*}
    to get a wet $N$-cluster satisfying
    \begin{align}\label{wet truncated cluster estimates}
        P(\F(N+1)) &\le (2 + O(\delta))\h\Big(\bigcup_{h=1}^N \E(h) \cap\{u = r_0\}\Big) \nonumber\\&+ \delta\H^{n-1}\Big( \pas\E(N+1) \cap\{u = r_0\}\Big) + P(\E^{r_0}(N+1))
    \end{align}
\end{remark}
The generalized minimizers we are interested in working with still have limited compactness results. As such we will require a truncation lemma for them as well. While the proof is the identical strategy, it is distinct enough that we include it as a separate lemma.
\begin{lemma}\label{truncation for wet cluster}
    Let $F$ be a closed set, $u(x) = {\rm dist}(F,x)$, and $\E$ an $(N+1)$-cluster. If it holds that 
    \begin{equation*}
        \sum_{h=1}^{N+1} \big|\E(h)\setminus F\big| \leq \a,
    \end{equation*}
    then there exists $r_0\in [0,13(n+1)\alpha^{\frac{1}{n+1}}]$ such that the $(N+1)$-cluster $\E'$ defined by 
    \begin{align*}
        \E'(h) = \E(h) \cap\{u\leq r_0\}, && 1\leq h\leq N+1
    \end{align*} satisfies
    \begin{equation*}
        P_{\rm wet}^*(\E') \leq P^*_{\rm wet}(\E) - \frac{d(\E,\E')}{4\alpha^{\frac{1}{n+1}}}.
    \end{equation*}
\end{lemma}
\begin{proof}
    In the event that $\E(h)\subset \{u \le 13(n+1)\alpha^{\frac{1}{n+1}}\}$, we set $r_0 = 13(n+1)\alpha^{\frac{1}{n+1}}$. Otherwise, suppose for the sake of contradiction that for all $r \in [0,13(n+1)\alpha^{\frac{1}{n+1}}]$,
    \begin{equation*}
        \ps(\E) < \ps(\E^r) + \frac{m(r)}{4\alpha^{\frac{1}{n+1}}},
    \end{equation*}
    where $m(r) = d(\E,\E^r)$ and $\E^r(h)$ is the truncated cluster exactly as in the proof of Lemma \ref{truncation lemma} (with $\E^r(0) = \rn\setminus\bigcup_{h=1}^{N+1}\E^r(h)$, according to our convention). 
    This would imply that
    \begin{equation}\label{gen trunc contradiction hyp}
        \sum_{h=0}^{N} P(\E(h)) - P(\E^r(h))< \frac{m(r)}{4\alpha^{\frac{1}{n+1}}}.
    \end{equation}
    The coarea formula and the fact that clusters are Caccioppoli partitions of $\R^{n+1}$ will imply that 
    \begin{align}\nonumber
        \sum_{h=0}^{N} P(\E(h)) - P(\E^r(h)) &= \sum_{h=0}^N P(\E(h);\{u > r\}) - 2\sum_{h=1}^N \h(\E(h)\cap \{u=r\})\\\nonumber
        &\ge \frac{1}{2}\sum_{h=1}^{N+1} P(\E(h);\{u > r\})  - 2m'(r)\\\nonumber
        &= \frac{1}{2}\sum_{h=1}^{N+1} P(\E(h)\cap\{u > r\}) + 3m'(r)\\\label{gen trunc pre iso}
        &\geq \frac{1}{2}P(\bigcup_{h=1}^{N+1}\E(h)\cap\{u>r\}) + 3m'(r).
    \end{align}
    Applying the non sharp isoperimetric inequality to \eqref{gen trunc pre iso}
    \begin{equation*}
        \sum_{h=0}^{N} P(\E(h)) - P(\E^r(h)) \geq \frac{1}{2}m(r)^\frac{n}{n+1} + 3m'(r).
    \end{equation*}
    Plugging this back into \eqref{gen trunc contradiction hyp} and using the fact that $m$ is decreasing with $m(0) \leq \alpha$ we observe that
    \begin{equation*}
        \frac{1}{2}m(r)^\frac{n}{n+1} + 3m'(r) < \frac{1}{4}m(r)^\frac{n}{n+1}
    \end{equation*}
    or equivalently,
    \begin{equation*}
        m(r)^\frac{n}{n+1} <  -12m'(r)
    \end{equation*}
    This implies that 
    \begin{equation*}
        (m(r)^\frac{1}{n+1})' < -\frac{1}{12(n+1)}.
    \end{equation*}
    Integrating this inequality over $[0,13(n+1)\alpha^{\frac{1}{n+1}}]$ will show that
    \begin{equation*}
        m(13(n+1)\alpha^{\frac{1}{n+1}})^\frac{1}{n+1} - m(0)^\frac{1}{n+1} < -\frac{13}{12}\a^\frac{1}{n+1}.
    \end{equation*}
    Again using the monotonicity of $m$, we see by the above
    \begin{equation*}
        m(0)^\frac{1}{n+1}\leq \a^\frac{1}{n+1} < \frac{13}{12}\a^\frac{1}{n+1} < m(0)^\frac{1}{n+1}.
    \end{equation*}
    This is a contradiction, and thus we conclude the proof of the lemma.
\end{proof}

Finally, we state the nucleation lemma as seen in \cite{Maggi_Book} Section 29.3. This is also necessary for all compactness results.
\begin{lemma}\label{nucleation lemma}
    If $E$ is a set of finite perimeter with $0 < |E| < \infty$ and 
    \begin{equation*}
        \e \le \min\bigg\{|E|,\frac{P(E)}{2(n+1)c(n)}\bigg\},
    \end{equation*}
    then there exists a finite family of points $I\subset\R^{n+1}$ such that 
    \begin{align*}
        \Big|E\setminus \bigcup_{x\in I} &B_2(x) \Big| < \e,\\
        \Big|E\cap B_1(x)\Big|&\ge \Big(c(n)\frac{\e}{P(E)}\Big)^{n+1}, \,\,\,\,\,\,\,\forall x\in I.
    \end{align*}
    Moreover, $|x-y| >2$ for $x,y\in I$ and $x\neq y$, and
    \begin{equation*}
        \#I \leq |E|\Big(\frac{P(E)}{\e c(n)}\Big)^{n+1}.
    \end{equation*}
\end{lemma}
\begin{proof}
    See \cite{Maggi_Book} Section 29.3.
\end{proof}

\subsection{Preliminary results.} 
Here, we collect a few lemmas that will be essential for defining and utilizing the competitors defined in the upcoming sections.

For $A\subset\rn$ and $r>0$ we write $A_r = A\cap\pa B_r$, regarded as a subset of the sphere $\pa B_r$; perimeters, reduced boundaries, and the notation $P(A_r)$ for subsets of $\pa B_r$ are always understood relative to $\pa B_r$.
\begin{lemma}\label{Perimeter sections}
If $E\subset\R^{n+1}$ is a set of finite perimeter, then $\pas (E_r) \stackrel{\mathcal{H}^{n-1}}{=} (\pas E)_r$ for almost every $r\in\R$
\end{lemma}
\begin{proof}
\textit{Step One}: First, we show that $E_r$ is a set of finite perimeter in $\pa B_r$.

\medskip
\noindent
Given a sequence of mollifiers $\p_\epsilon$, let $u_\e = 1_E\star\p_\e$ and $u = 1_E$. Then $u_\e \to 1_E$ in $L^1_{\rm loc}(\R^{n+1})$. It follows that, for $R\in(0,\infty)$, we may choose an $\e_h$ such that
\begin{equation}\label{convergence in measure for sofp}
\int_0^R r^{n}\int_{\pa B_1}|u_{\e_h}(r\sigma)-u(r\sigma)|d\h dr < \frac{1}{h^2}.
\end{equation}
Applying Chebyshev's inequality to \eqref{convergence in measure for sofp} implies that 
\begin{equation*}
r^{n}\int_{\pa B_1}|u_{\e_h}(r\sigma)-u(r\sigma)|d\h \to 0 \text{ in measure}.
\end{equation*}
Moreover, up to a subsequence we have that 
\begin{equation}\label{l1 convergence of sections}
    r^{n}\int_{\pa B_1}|u_{\e_h}(r\sigma)-u(r\sigma)|d\h \to 0 \text{ for almost every } r\in(0,R),
\end{equation}
implying that for almost every $r\in\R^+$ that $u_{\e_h} \to u$ in $L^1(\pa B_r)$. 

\noindent
Now choose $T \in \X(\pa B_r)$ such that $|T|\leq 1$. Then by \eqref{l1 convergence of sections} we have 
\begin{align}\label{section is fp}
|\!\int_{E_r} {\rm div}^{\pa B_r}\, T(r\sigma)| &= \lim_{h\to\infty}|\int_{\pa B_r} u_{\e_h}(r\sigma){\rm div}^{\pa B_r}\,T(r\sigma)|&\nonumber\\
&=\lim_{h\to\infty}|\int_{\pa B_r}{\rm div}^{\pa B_r}\,u_{\e_h}(r\sigma)T(r\sigma) - \nabla^{\pa B_r} u_{\e_h}(r\sigma)\cdot T(r\sigma)|.&
\end{align}
As $u_{\e_h}T\in\X(\pa B_r)$ and $\pa B_r$ is boundaryless, we further see that
\begin{equation*}
\int_{\pa B_r}{\rm div}^{\pa B_r}\,u_{\e_h}(r\sigma)T(r\sigma) = \int_{\pa B_r}u_{\e_h}(r\sigma)T(r\sigma)\cdot\textbf{H}^{\pa B_r} = 0.
\end{equation*}
Combining this with \eqref{section is fp} and $|T|\le1$, we have  \begin{equation}\label{d-1 perimeter bound}
\sup|\!\int_{E_r} {\rm div}^{\pa B_r}\, T(r\sigma)| \leq \liminf_{h\to\infty} \int_{\pa B_r} |\nabla^{\pa B_r} u_{\e_h}|.
\end{equation}
Finally, integrating \eqref{d-1 perimeter bound} and applying the coarea formula yields
\begin{equation*}
\int_0^\infty \sup|\!\int_{E_r} {\rm div}^{\pa B_r}\, T(r\sigma)| \leq \int_0^\infty P(E_r) \leq P(E) < \infty,
\end{equation*}
implying that $E_r$ is finite perimeter for almost every $r$.

\medskip

\noindent \textit{Step Two}: We now show that $\pas (E_r) \stackrel{\mathcal{H}^{n-1}}{=} (\pas E)_r$.

\medskip\noindent Let $u(x) = |x|$ and $\S = \{x \in \pas E: \nu_E(x) = \pm\nabla u\}.$ By the coarea formula, we have that 
\begin{equation*}
\int_\R \H^{n-1}(\S_r) \,dr= \int_\R\,dr\int_{\S\cap(\pas E)_r} = \int_{\pas E} 1_{\S}\sqrt{1 - (\nu_E\cdot\nabla u)^2} = 0,
\end{equation*}
implying that for almost every $r$, $\H^{n-1}(\Sigma_r)=0$. Now if we define $M = \pas E \setminus \S$, it suffices to show that for almost every $r\in\R$,
\begin{equation}\label{measure for sections}
\mu_{E_r} = \frac{(\nabla^{\pas E} u)\nu_E(x)}{|(\nabla^{\pas E} u)\nu_E(x)|}\H^{n-1}\llcorner(M_r).
\end{equation}
To this end, consider $T \in \X(\pa B_r)$ and $\phi \in C^1_c(\R)$. Then define $S \in C^1_c(\R^{n+1},\R^{n+1})$ by 
\begin{equation*}
    S(r\sigma) = \phi(r)T(\sigma).
\end{equation*}
By the divergence theorem for rectifiable sets, we see that
\begin{align}\label{divergence thm for sections}
\int_0^\infty \phi(r) \int_{E_r} r^n{\rm div}^{\pa B_r}\,T d\H^{n} dr = \int_E {\rm div} \,S \,d\H^{n+1} = \int_{\pas E} S\cdot\nu_E(x)d\h\nonumber\\
= \int_{M}S\cdot(\nabla^{\pas E} u)\nu_E(x) = \int_0^\infty \phi(r) \int_{M_r} r^nT\cdot\frac{(\nabla^{\pas E} u)\nu_E(x)}{|(\nabla^{\pas E} u)\nu_E(x)|} d\H^{n} dr.
\end{align}
Immediately, \eqref{measure for sections} follows from \eqref{divergence thm for sections}. 
\end{proof}

\begin{remark}
    We briefly recall the notion of essential connectedness. A set $S$ essentially disconnects a set $T$ if there exists a nontrivial Borel partition $\{T_1, T_2\}$ such that
    $$T^{(1)}\cap\pa^eT_1 \cap \pa^eT_2\hnsub S.$$
    
\end{remark}

\begin{lemma}\label{essential partitions}
If $S \subset \SS^{n}$ satisfies $\H^{n-1}(S) < \infty$, then there exists an essential partition $\{A_i\}$ of $\SS^n\setminus S$ induced by $S$ such that 
    \begin{equation}\label{partition perimeter}
    \sum_{i} P(A_i;\SS^n) \leq 2\H^{n-1}(S)\,.
    \end{equation}
\end{lemma}
\noindent
The proof is a repetition of the argument given for Theorem 2.1 of \cite{maggi2023plateauborderssoapfilms}, with the only change being the ambient space that it is done in. As this changes nothing about the argument, we do not repeat it here.


\begin{lemma}\label{spherical isoperimetry}
    Let $n\geq 2$. Given a spherical cap $S$ and $J\subset S$ with $\H^{n-1}(J) < \infty$, let $\{A_h\}_{h\ge1}$ denote the essentially connected components of $ S\setminus J$ from Lemma \ref{essential partitions}, ordered so that $\h(A_h)\geq \h(A_{h+1})$. Then for every $\alpha\subset\N\setminus\{1\}$ it holds that
    $$\h\Big(S\setminus \big(A_1 \cup \bigcup_{h\in\alpha} A_h\big)\Big) \leq C_S\H^{n-1}(J)^{\frac{n}{n-1}}.$$
\end{lemma}

\noindent This is Lemma 9 from \cite{delellis2014directapproachplateausproblem} with enough modification that we re-prove it here. 
\begin{proof}
    By the relative isoperimetric inequality, for any $A\subset S$ $$\min\bigg\{\h(S\setminus A), \h(A)\bigg\}\leq C_S\H^{n-1}(\pas A)^\frac{n}{n-1}.$$
    Thus by the ordering property, for $h\geq 2$ $$\h(A_h) \leq C_S\H^{n-1}(\pas A_h)^\frac{n}{n-1}.$$
    By the superadditivity of $t \mapsto t^\frac{n}{n-1}$,

    \begin{equation}\label{sphere iso superadditivity}
    \sum_{h\in\N\setminus(\{1\}\cup\a)} \h(A_h) \leq C_S\big(\sum_{h\in\N\setminus(\{1\}\cup\a)} \H^{n-1} (\pas A_h)\big)^\frac{n}{n-1}.
    \end{equation}
    By Lemma \ref{essential partitions}, 
    \begin{equation}\label{sphere iso ess part est}
        \sum_{h\in\N\setminus(\{1\}\cup\a)} \H^{n-1} (\pas A_h) \le 2\H^{n-1}(J).
    \end{equation}
    Moreover, $S = J \cup \bigcup_{h=1}^\infty A_h$ and $\h(J) = 0$
    imply that 
    \begin{equation}\label{essential partitions of a sphere}
    \sum_{h\in\N\setminus(\{1\}\cup\a)} \h(A_h) = \h\Big(\bigcup_{h\in \N\setminus(\{1\}\cup\a)} A_h\Big) = \h\Big(S\setminus \big(A_1\cup\bigcup_{h\in\a}A_h\big)\Big).
    \end{equation}
    \eqref{sphere iso superadditivity}, \eqref{sphere iso ess part est} and \eqref{essential partitions of a sphere} collectively imply the result.
\end{proof}

\subsection{Cup competitors.}\label{subsec:cup}
Given a wet $N$-cluster $\E$, consider the ball of radius $r$ centered at a point $x$. By Lemmas \ref{Perimeter sections} and \ref{essential partitions}, for almost every $r$ the set $(\pas\E(N+1))_r$ has finite $\H^{n-1}$-measure and induces an essential partition $\{A_i\}$ on $\pa B_r$ such that \eqref{partition perimeter} holds. Without loss of generality, we may assume that $\h(A_i) \geq \h(A_{i+1})$. The cup competitor corresponding to $\E$ and $B_r$ is defined as follows:

\medskip

\noindent 
    \begin{enumerate}[label={\rm(\roman*).}]
        \item If $A_1 \hnsub \E(h)$ for $h\neq N+1$, then we define the set 
        \begin{equation}\label{Y defn1}
            Y = \bigcup\{A_i:\h(A_i \cap\E(N+1)) = 0, i\neq 1\}.
        \end{equation}
        We define the cup competitor $\F_{\eta,x,r}$, with respect to the ball $B_r(x)$, by \begin{align*}
        \F_\eta(N+1) &= \Big(\E(N+1) \setminus B_r(x)\Big) \cup N_\eta(Y)\\
        \F_\eta(h) &= \Big(\E(h)\cup B_r(x)\Big)\setminus N_\eta(Y)\\
        \F_\eta(i) &= \E(i)\setminus B_r(x) \text{ for }i = 0,\dots h-1, h+1,\dots N.
        \end{align*}
        In the case where the ball $B_r(x)$ is clear, we shorten $\F_{\eta,x,r}$ to $\F_\eta$. 
        
        \medskip
        
        \item If instead it holds that $A_1 \hnsub \E(N+1)$, then we define \begin{equation}\label{Y defn2}
            Y = \bigcup\{A_i:A_i \hnsub \E(N+1), i\neq 1\}.
        \end{equation} The cup competitor $\F_\eta$ is now defined to be 
        \begin{align*}
        \F_\eta(N+1) &= \Big(\E(N+1) \cup B_r(x)\Big) \setminus N_\eta(Y)\\
        \F_\eta(0) &= \Big(\E(0)\setminus B_r(x)\Big)\cup N_\eta(Y)\\
        \F_\eta(i) &= \E(i)\setminus B_r(x) \text{ for }i = 1,\dots N.
        \end{align*}
    \end{enumerate}
This leads us to the following lemma.
\begin{lemma}\label{cup competitor main}
    If $\E$ is a wet $N$-cluster, $x\in\R^{n+1}$, and $\{A_i\}$ is the (ordered) partition of $\pa B_r$ given by Lemma \ref{essential partitions}, then for almost every $r\in\R$, $\F_{\eta}$ is a wet $N$-cluster. \\
    Moreover,
    \begin{enumerate}[label={\rm(\roman*).}]
        \item If $\h(A_1 \cap \E(N+1)) = 0$, then 

        \begin{equation}\label{cup competitor main lemma 1}
            \limsup_{\eta\to0}\h(\pas \F_\eta(N+1)\cap B_r) \leq \h\big(\pa B_r\setminus(\E(N+1) \cup A_1)\big)
        \end{equation}
        \item If $A_1\hnsub\E(N+1),$ then

        \begin{equation}\label{cup competitor main lemma 2}
            \limsup_{\eta\to0}\h(\pas \F_{\eta}(N+1)\cap B_r) \leq \h\big((\E(N+1)\cap \pa B_r)\setminus A_1\big)
        \end{equation}
    \end{enumerate}
\end{lemma}
\begin{proof}
    Let $\F = \F_{\eta,x,r}$, where $r$ is chosen depending on $x$ such that Lemma \ref{Perimeter sections} holds. We prove the lemma in the case that $\h(A_1\cap\E(N+1)) = 0$. The other case follows via the identical logic. We first show that $\F$ is a wet $N$-cluster. It is clear by construction that $(\pas\F(h))\setminus\pa B_r \subset \pas\F(N+1)$ for any $h$, so it suffices to consider $y\in\pa B_r$. Moreover, given that our considerations ignore sets of $\H^{n-1}$ measure 0, we may further assume without loss of generality that $i\neq 1$ and $y\in A_i^{(1)}\subset\E(j)$ for $j = h$ or $N+1.$ If $j = N+1$, the fact that cup competitors are constructed via essential partitions, Lemma \ref{Perimeter sections} and Federer's theorem imply that $A_i^{(1)}\hnsub \E(N+1)^{(1)}.$ Then, \cite{Maggi_Book} Theorem 16.3 implies that $A_i^{(1)}\hnsub \pas\F(N+1)$. If $j = h$, then identically $A_i \hnsub \E(h)^{(1)}$ so we consider sets defined with respect to $\pa B_r$ by
    $$\mathcal{N}_{\p}(x) = \{y + t\nu_{\pa B_r}(y): y\in B^{\pa B_r}_\p(x), t\in(-\p,\p)\}$$
    We observe that for any $\sigma > 0$, there exists a $p$ such that for $\rho < p$ it holds that 
    \begin{equation*}
        \frac{\h\big(B^{\pa B_r(x)}_\p(y)\cap A_i\big)}{\h\big(B_\p^{\pa B_r(x)}(y)\big)} \geq 1-\sigma.
    \end{equation*}
    Therefore we observe by \cite{Maggi_Book} Exercise 5.19 and the above that
    \begin{align*}
        \frac{1}{2} \geq \frac{|\NN_\p(y)\cap \F(N+1)|}{|\NN_\p(y)|} \geq \frac{1}{2}(1-O(\p))(1-\sigma).
    \end{align*}
    This implies that $y\in\pa^e\F(N+1)\hneq \pas\F(N+1)$, which completes the proof that $\F$ is a wet $N$-cluster.\\
    Next observe that 
    \begin{align}
        \h\big(\pas\F(N+1) \cap B_r(x)\big) &= \h\big(\pas\F(N+1)\cap\pa B_{r-\eta}(x)\big) \nonumber\\&+ \h\big(\pas\F(N+1)\cap(B_r(x)\setminus \overline{B_{r-\eta}})\big). \nonumber
    \end{align}
    Up to a translation, we may assume without loss of generality that $x = 0$. The area formula implies that
    \begin{equation}\label{area formula for cup competitors}
        \h\big(\pas\F(N+1) \cap \pa B_{r - \eta}\big) = (\frac{r-\eta}{r})^n\h\big(\frac{r}{r-\eta}(\pas\F(N+1) \cap \pa B_{r - \eta})\big).
    \end{equation}
    By \eqref{Y defn1}, $\frac{r}{r-\eta}(\pas\F(N+1) \cap \pa B_{r - \eta}) = Y$, so \eqref{area formula for cup competitors} implies 
    \begin{equation}\label{estimate on Y}
        \h\big(\pas\F(N+1) \cap \pa B_{r - \eta}\big) = (\frac{r-\eta}{r})^n\h(Y).
    \end{equation}
    Next, if $u(x) = |x|$ then we claim that 
    \begin{equation}\label{coarea factor is 1 for cup}
        |\nabla^{\pas\F(N+1)}u| = 1 
    \end{equation}
    As $|\nabla u| = 1,$ it suffices to show that $(\nu_{\F(N+1)}\cdot \nabla u) = 0$. Indeed, consider $y \in \pas\F(N+1) \cap (B_r(x)\setminus \overline{B_{r-\eta}})$. By construction of $\F$, $\pas F(N+1)$ will contain a radial segment containing $y$. This implies that $\nu_{\F(N+1)}$ is perpendicular to the line connecting $0$ to $y$, however $\nabla u$ is parallel to that line, proving \eqref{coarea factor is 1 for cup}. Thus the coarea formula implies that \begin{equation*}
        \h\big(\pas\F(N+1)\cap(B_r(x)\setminus \overline{B_{r-\eta}})\big) = \int_{r-\eta}^r \H^{n-1}\big(\pas\F(N+1)\cap\pa B_{\p}(x)\big)d\p.
    \end{equation*}
    Applying the area formula then gives that 
    \begin{equation}\label{Estimates for the radial component for cup}
        \int_{r-\eta}^r \H^{n-1}\big(\pas\F(N+1)\cap\pa B_{\p}(x)\big)d\p = P(Y)\int_{r-\eta}^r (\frac{\p}{r})^{n-1}d\p \leq \eta P(Y).
    \end{equation}
    As $P(Y) < \infty$ for a.e. $r$, combining \eqref{estimate on Y} with \eqref{Estimates for the radial component for cup} gives \eqref{cup competitor main lemma 1}.
\end{proof}
\begin{remark}
    When constructing a cup competitor $\F_{\eta,x,r}$ given an essential partition of $\pa B_r(x), \{A_i\}$, the identical construction and estimates will work even if $\{A_i\}$ is unordered. We often require the ordering for the later steps, so it is notationally convenient to assume that $\{A_i\}$ is ordered.
\end{remark}

\subsection{Cone competitors.}
Suppose again that we have a wet $N$-cluster and a ball $B_r(x)$, which induces an ordered essential partition $\{A_i\}$ via Lemma \ref{essential partitions}. We define the set \begin{equation*}
    Z(h) = \bigcup\Big\{A_i : A_i\hnsub \E(h)\Big\}
\end{equation*} 
Then we define the cone competitor $\F_{x,r}$ (again shortened to $\F$) by
\begin{equation}
    \F(h) = \big(\E(h)\setminus B_r(x)\big)\cup\big\{ty + (1-t)x: t\in[0,1], y\in \pa B_r(x) \cap Z(h)\big\}
\end{equation}
This brings us to the following lemma
\begin{lemma}\label{Cone competitors basic info}
    If $\E$ is a wet $N$-cluster, $x\in\R^{n+1}$ and $B_r(x)$ is chosen such that \begin{equation*}
        \H^{n-1}\big(\pas\E(N+1)\cap\pa B_r(x)\big) < \infty,
    \end{equation*}then the cone competitor $\F_{x,r}$ is a wet $N$-cluster satisfying the estimate
    \begin{equation}\label{cone competitor estimate}
        \H^{n}\big(\pas\F(N+1) \cap B_r(x)\big) \leq \frac{r}{n}\H^{n-1}\big(\pas\F(N+1)\cap\pa B_r(x)\big)
    \end{equation}
\end{lemma}
\begin{proof}
    As $\E$ is itself a wet $N$-cluster, by our choice of $r$ satisfying Lemma \ref{Perimeter sections} we see 
    \begin{equation*}
        \pas\F(h)\cap\pa B_r(x) = \pas\E(h)\cap\pa B_r(x) \stackrel{\H^{n-1}}{\subset} \pas\E(N+1)\cap\pa B_r(x) = \pas\F(N+1)\cap\pa B_r(x)
    \end{equation*}
    As $\pas\F(h)\cap\pa B_\p(x)$ is a diffeomorphic image of $\pas\F(h)\cap\pa B_r(x)$ via scaling, we have 
    \begin{equation}\label{cone competitor containment condition}
        \H^{n-1}\big((\pas\F(h)\setminus\pas\F(N+1))\cap\pa B_\p(x) \big) = 0
    \end{equation}
    If we again assume by translating that $x=0$ and let $u(y) = |y|$, the verbatim reasoning of \eqref{coarea factor is 1 for cup} implies that 
    \begin{align}\label{coarea factor is 1 for cone}
        |\nabla^{\pas\F(h)\setminus\pas\F(N+1)} u| &= 1 \text{ for } y\in\pas\F(h)\setminus\pas\F(N+1)\\
        \label{coarea factor is 1 for cone pt 2}
        |\nabla^{\pas\F(N+1)}u| &= 1 \text{ for } y\in\pas\F(N+1).
    \end{align}
    Applying the coarea formula to \eqref{cone competitor containment condition}, we see by \eqref{coarea factor is 1 for cone} that $\F$ is a wet $N$-cluster.\\
    By using \eqref{coarea factor is 1 for cone pt 2}, the coarea formula implies that
    \begin{align*}
        \h\big(\pas\F(N+1)\cap B_r(x)\big) &= \int_0^r\H^{n-1}(\pas\F(N+1)\cap\pa B_\p(x))\,d\p\\ &= \int_0^r\big(\frac{\p}{r}\big)^{n-1}\H^{n-1}\Big(\big(\frac{r}{\p}\big)\pas\F(N+1)\cap \pa B_{\p}\Big)\,d\p \\&= \H^{n-1}\big(\pas\F(N+1)\cap\pa B_r\big)\int_0^r  (\frac{\p}{r})^{n-1}\,d\p\\ &= \frac{r}{n}\H^{n-1}\big(\pas\F(N+1)\cap\pa B_r\big),
    \end{align*}
    which is exactly \eqref{cone competitor estimate}. 
\end{proof}

\subsection{Slab competitors.}\label{subsec:slab}
Using a bi-Lipschitz deformation, we may construct slab competitors from cup competitors. We consider the slab 
\begin{equation*}
    S_{\tau,r}^\nu(x) = \{y\in B_r(x): |(y-x)\cdot\nu| < \tau r\}
\end{equation*} 
As in \cite{King_2022}, we have the existence of a bi-Lipschitz map $\Phi_r:\R^{n+1} \to \R^{n+1}$ such that 
\begin{align*}
    \{\Phi_r\neq\,{\rm Id}\}\subset B_{2r}(x), && \Phi_r(B_{2r}(x)) = B_{2r}(x), && \Phi_r(\pa S_{\tau,t}^\nu(x)) = \pa B_t(x)\,\,\,\forall t\in(0,\frac{3r}{2}),
\end{align*}
and $\Lip(\Phi_r)$, $\Lip(\Phi^{-1}_r)$ depend only on $n$ and $\tau$. If $\E$ is a wet $N$-cluster, consider the cluster $\Phi_r(\E)$. We construct the cup competitor $\F_{\eta,x,r,A}$, where $A$ is now an \textit{arbitrary} element of the essential partition from Lemma \ref{essential partitions}. Finally, we take $\Sc_{\eta, \tau, x, r, A}^\nu = \Phi_r^{-1}(\F_{\eta,x,r})$. When the parameters are clear, we shorten $\Sc_{\eta, \tau, x, r, A}^{\nu}$ to $\Sc_{\eta,A}$ and $\Sc_{\eta, \tau, x, r, A}^\nu(N+1)$ to $E_A$.
\begin{lemma}\label{slab competitor main}
    If $\E$ is a wet $N$-cluster, $\nu\in \SS^{n}$, $x\in\R^{n+1}$ and $\tau \in (0,1)$, then for almost every $r\in \R$ the slab competitor $\Sc_{\eta, \tau, x, r, A}^\nu$ is a wet $N$-cluster.\\
    Moreover,
    \begin{enumerate}[label={\rm(\roman*).}]
        \item If $\h\big(A \cap \Phi_r(\E(N+1))\big) = 0$, then 
        \begin{equation}\label{slab competitor main estimate 1}
            \limsup_{\eta\to0}\h\big(\pas E_A\cap S^\nu_{\tau,r}(x)\big) \leq C(n,\tau)\h\Big(\pa S_{\tau,r}^\nu(x)\setminus\big(\E(N+1) \cup \Phi_r^{-1}(A)\big)\Big)
        \end{equation}
        \item If $A\hnsub\Phi_r(\E(N+1))$, then
        \begin{equation}\label{slab competitor main estimate 2}
            \limsup_{\eta\to0}\h\big(\pas E_A\cap S_{\tau,r}^\nu(x)\big) \leq C(n,\tau)\h\Big(\E(N+1)\cap\pa S_{\tau,r}^\nu\setminus \Phi_r^{-1}(A)\Big)
        \end{equation}
    \end{enumerate}
\end{lemma}
 
\begin{proof}
    The first observation to make is that because for any $r$, $\Phi_r$ is bi-Lipschitz, Proposition 17.1 from \cite{Maggi_Book} will imply that $\Phi_r(\E)$ is a wet $N$-cluster. The construction of $\Phi_r$ involves taking the cutoff of a function $\Psi$ that satisfies $\Psi(\pa S^\nu_{\tau,t}) = \pa B_t$ for all $t$. Letting $\G = \Psi(\E)$, each chamber of $\G$ will be of locally finite perimeter as $\Psi$ is locally Lipschitz. Thus for an $r$ where Lemma \ref{Perimeter sections} holds, it follows immediately that $\pas \Phi_r(\E)_r = (\pas\Phi_r(\E))_r$ as $\Phi_r$ and $\Psi$ agree on a neighborhood of $\pa B_r(x)$. Thus we immediately conclude that the cup competitor to $\Phi_r(\E)$, $\F_{\eta,x,r,A}$ is a wet $N$-cluster satisfying the estimates of Lemma \ref{cup competitor main}. Again because $\Phi_r$ is bi-Lipschitz, $\Phi_r^{-1}(\F_{\eta,x,r,A}) = \Sc^\nu_{\eta,\tau,x,r,A}$ is a wet $N$-cluster. \\
    If $\h\big(A\cap\Phi_r(\E(N+1)\big) = 0$, then $\F$ satisfies \eqref{cup competitor main lemma 1}. Then, \eqref{slab competitor main estimate 1} follows immediately by applying the fact that
    \begin{equation*}
        \h(f(E))\leq\lip(f)^n\h(E)
    \end{equation*}
    to \eqref{cup competitor main lemma 1}, where we see that $C(n,\tau) = \lip(\Phi_r)^n\lip(\Phi_r^{-1})^n$. When \eqref{cup competitor main lemma 2} holds instead, the reasoning is identical to show \eqref{slab competitor main estimate 2} will hold.
\end{proof}

\begin{remark}\label{multi component wet slab}
    We may allow $A$ to be a union of indecomposable components $A_1 \cup A_2$ so long as there exists an $h$ such that $\h((A_1 \cup A_2) \setminus\E(h)) = 0$. This will prove useful in determining the structure of generalized minimizers to the wet isoperimetric problem. As this version is a bit more messy notationally to work with, we only use it when strictly necessary.
\end{remark}

\section{Existence of generalized minimizers} 
In this section, we prove Theorem \ref{existence thm main} by utilizing the various tools developed in Section \ref{sec:competitors}. The overview of the proof is as follows. In step one, we show that we may modify a $\psi_\e(v)$-minimizing sequence of wet $N$-clusters $\E_j$ with a limit cluster $\E$. Moreover, the measures $\h\llcorner \pas\E_j(N+1)$ will have a weak star limit $\mu.$ In step two, we show that the boundaries of the chambers of the limit cluster are contained in the $K = \text{supp}(\mu).$ In step three, we provide a geometric argument to show $\theta_*(\mu)\geq 2$ on $\pas\E(i)\setminus\pas\E(N+1)$ for any $i\neq N+1$. In step four, we utilize cup and cone competitors to show that $\mu = \theta\h\llcorner K$ for $\h$-rectifiable $K$ and $\theta$ upper semicontinuous. In step five, we use slab competitors to show $\theta_*(\mu)= 2$ on $\pas\E(i)\setminus\pas\E(N+1)$ and $\theta_*(\mu) = 1$ on $\pas\E(N+1)$. In step six, we rule out interior collapsing with the multi-component slab competitors from remark \ref{multi component wet slab}.
\begin{proof}
    [Proof of Theorem \ref{existence thm main}] 
  
    \medskip
  
    \noindent {\it Step one}: In this step, we show that for a minimizing sequence of wet $N$-clusters, we may modify the clusters such that they have a limiting $(N+1)$-cluster. This is done using minor modifications of the classical nucleation, truncation, and volume fixing variation process found in Chapter 29 of \cite{Maggi_Book}. Let $s_0 = (2\frac{\e+\sum_{k=1}^N v_k}{\omega_{n+1}})^{\frac{1}{n+1}}$. Define the space \begin{equation*}
      V= \{a\in\R^{N+2}:\sum_{h=0}^{N+1}a(h) = 0\}.
    \end{equation*} By step one of Section 29.7 in \cite{Maggi_Book}, there exist $\{x(h)\}_{h=0}^{N+1}$ and $r_*,\, C, \, \e_1 > 0$ such that there exist C$^1$ functions 
    \begin{equation*}
      \Psi_k: \big((-\e_1,\e_1)^{N+2}\cap V\big) \times\rn\to\rn,
    \end{equation*}
    where $\Psi_k(a,\cdot)$ is a diffeomorphism satisfying
    \begin{align}
        \big\{x\in\R^{n+1}:\Psi_k(a,x) \neq x\big\}&\subset\subset \bigcup_{h=1}^{N+1} B_{s_0}(x_k(h)),\label{support of volume fixing variation}\\
        \Big|\Psi_k(a,\E_k(h))\Big| &= \big|\E_k(h)\big| + a(h),\label{volume fixing for volume fixing variations}\\
        \Big|\h\big(\Psi_k(a,\S)\big) - \h(\S)\Big| &\leq C\h(\S)\sum_{h=0}^{N+1}|a(h)|\label{perimeter estimates for volume fixing variations}
    \end{align}
    for any $\h$-rectifiable set $\S$. Then for sufficiently small $\e_0$, by Lemma \ref{nucleation lemma} we deduce the existence of a sequence of finitely many points $\{x_k(h,i)\}_{i=1}^{L(h,k)}$ with the property that
    \begin{equation*}
        \Big|\E_k(h)\setminus\bigcup_{i=1}^{L(h,k)} B_2(x_k(h,i))\Big|<\e_0
    \end{equation*}
    and $L(h,k) < C$ uniformly in $h$ and $k$. Next, we define the closed sets
    \begin{equation*}
        F_k = \bigcup_{h=1}^{N+1}\Big(\overline{B}_{s_0}(x_k(h)) \cup\bigcup_{i=1}^{L(h,k)} \overline{B}_2(x_k(h,i))\Big).
    \end{equation*}
    Then applying Lemma \ref{truncation lemma} with $\E_k$ and $F_k$ (so that $u_k = \dist(F_k,\cdot)$), we get new clusters $\E_k'$ and corresponding $r_k$ satisfying \eqref{truncation estimates}. Define $\F_k$ as in Remark \ref{wet truncated cluster remark} with $\delta_k$ chosen small enough that 
    \begin{equation*}
        P(\F_k(N+1)) \leq P(\E_k'(N+1)) + 2\h\Big(\bigcup_{h=1}^N \E_k(h) \cap\{u_k = r_k\}\Big)  + \frac{1}{k}.
    \end{equation*}
    As $\E_k'$ satisfies \eqref{truncation estimates}, we gain the stronger information
    \begin{equation}\label{capped truncation estimate}
        P(\F_k(N+1)) \leq P(\E_k(N+1)) + \frac{1}{k} - \frac{d(\E_k,\E_k')}{4\e_0^{\frac{1}{n+1}}}
    \end{equation}
    If $\delta_k$ and $\e_0$ are sufficiently small, we may apply the volume fixing diffeomorphism $\Psi_k$ to $\F_k$ to get a new cluster $\E''_k$. Using \eqref{volume fixing for volume fixing variations}, we can ensure that 
    \begin{equation}
        |\E_k''(h)| = |\E_k(h)| \,\,\,{\rm for}\,\, 1\leq h\leq N+1.
    \end{equation}
    Moreover, applying \eqref{perimeter estimates for volume fixing variations} to the set $\pas \F(N+1)$ alongside the triangle inequality yields the estimate
    \begin{equation}\label{truncation volume fixing}
        P(\E_k''(N+1)) \leq P(\F_k(N+1)) + 2CP(\F_k(N+1))\big(d(\E_k,\E_k')+d(\E_k',\F_k)\big)
    \end{equation}
    Again up to shrinking $\de_k$, we may ensure $d(\E_k,\E_k')\geq d(\E_k',\F_k)$. Using this in conjunction with \eqref{capped truncation estimate}, \eqref{truncation volume fixing} becomes
    \begin{equation}
        P(\E_k''(N+1)) \leq P(\E_k(N+1)) + \frac{1}{k} - \frac{d(\E_k,\E_k')}{4\e_0^{\frac{1}{n+1}}} + 8C\psi_\e d(\E_k,\E_k').
    \end{equation}
    Thus for a sufficiently small $\e_0$, $\E_k''$ is also a perimeter minimizing sequence. However, it also satisfies the containment condition $\E_k'' \subset \{u_k \leq r_k + \delta_k\}$. As we can assume without any loss of generality $\delta_k < (n+1)\e_0^{\frac{1}{n+1}}$, we have the stronger conclusion $\E_k'' \subset \{u_k \leq 12(n+1)\e_0^\frac{1}{n+1}\}$, implying that the sequence $\{\E_k''\}$ may be uniformly contained in some ball of radius $R$ about the origin.
  
    \medskip
  
    \noindent {\it Step two}: Step one implies that we may take a minimizing sequence $\E_k \to \E$ such that $\mu_{\E_k(N+1)} = \mu_k \weakstar \mu$. The remainder of the proof is dedicated to showing that $\mu$ satisfies the desired properties of Theorem \ref{existence thm main}, i.e. $\mu$ will correspond exactly with the relaxed wet energy for $\E$. In this step, we start to show that correspondence by showing
    \begin{equation}\label{support of limit measure}
        \bigcup_{h=1}^{N+1}\pas\E(h) \hnsub {\rm supp}(\mu).
    \end{equation}
    Indeed, if $x\in\pas\E(i)$ then we have that 
    \begin{equation*}
        \mu\Big(\overline{B_r(x)}\Big)\geq \limsup_{k\to\infty}P\Big(\E_k(N+1);\overline{B_r(x)}\Big) \geq \liminf_{k\to\infty} P\Big(\E_k(i);B_r(x)\Big)
    \end{equation*}
    As $\E_k \to \E$, lower semicontinuity of perimeter will imply that
    \begin{equation*}
        \mu\Big(\overline{B_r(x)}\Big) \geq \liminf_{k\to\infty} P\Big(\E_k(i);B_r(x)\Big) \geq P(\E(i);B_r(x)) > 0
    \end{equation*} proving \eqref{support of limit measure}.
  
    \medskip
  
    \noindent {\it Step three}:
    The next step is to get a lower bound on the densities of $\mu$ along the reduced boundaries of $\E$. It is immediate that if $x\in\pas\E(N+1)$, then $\t_n^*(\mu)(x)\geq 1$ by lower semicontinuity. We now show that if $x\in\pas\E(i)\setminus\pas\E(N+1)$, then $\t_n^*(\mu)(x) \geq 2.$
    If $x\in\pas\E(i)\cap\pas\E(j)$, then up to rigid motions we may assume that $x = 0$ and $\nu_{\E(i)}(0) = e_{n+1}.$ We define the half-spaces 
    \begin{align*}
        \{y_{n+1}\leq 0\} = H^-\\
        \{y_{n+1} \geq 0\} = H^+
    \end{align*}
    Then \cite{Maggi_Book} Theorem 15.5 and Corollary 15.8 will imply that 
    \begin{align}
        \big|(\E(i)\Delta H^-)\cap B_r(0)\big| &= o(r^{n+1})\label{lower half plane estimate}\\
        \big|(\E(j)\Delta H^+)\cap B_r(0)\big| &= o(r^{n+1})\label{upper half plane estimate}\\
        P\big(\E(i);B_r(0)\big) = P\big(\E(j);B_r(0)\big) &= \omega_nr^n + o(r^n) \label{blowup estimates}
    \end{align}
    Given $\sigma > 0$, by \eqref{lower half plane estimate} there exists $r_0(\sigma)$ such that for any $r < r_0$, it holds that
    \begin{equation}\label{half plane density estimate for lower bound}
        \big|(\E(i)\Delta H^-)\cap B_r(0)\big| \leq \frac{\om_n \s^2}{2}r^{n+1}.
    \end{equation}
    As $\E_k(i) \to \E(i)$, we have that $(\E_k(i)\Delta H^-) \cap B_r(0) \to (\E(i)\Delta H^-) \cap B_r(0)$, implying that we may choose $k_0(\sigma,r)$ such that for $k\geq k_0$, it holds via the triangle inequality that
    \begin{equation}\label{whole sequence half plane density estimate for lower bound}
        \big|(\E_k(i)\Delta H^-)\cap B_r(0)\big| \leq \om_n \s^2r^{n+1}.
    \end{equation}
    Now define the set
    \begin{equation*}
        Z_{k,\sigma}(i) = \{y \in D_r(0) : \exists z\in\E_k(i)^{(1)}: z_{n+1} < -\sigma r, \mathbf{p}(z) = y\},
    \end{equation*}
    where $D_r(0)$ is the $n-$dimensional disk about 0 and $\mathbf{p}(z)$ is the projection onto $D_r(0)$. \eqref{whole sequence half plane density estimate for lower bound} and some elementary geometry will imply that
    \begin{equation}\label{basic geometry volume bound}
        \big|\E_k(i)\cap (D_{r\sqrt{1-4\s^2}}\times[-2\s r,-\s r])\big| \ge \om_n\big((\sqrt{1-4\s^2})^n-\s\big)\s r^{n+1}.
    \end{equation}
    For brevity, we will define $\p = r\sqrt{1-4\s^2}$. By the coarea formula and \eqref{basic geometry volume bound}, we see
    \begin{equation*}
        \frac{1}{\s r}\int_{-2\s r}^{-\s r} \frac{\h\big(\E_k(i)\cap (D_{\p}(0)\times\{y\})\big)}{\h\big(D_\p(0)\big)}dy \ge 1 - \frac{\s}{(1-4\s^2)^\frac{n}{2}}.
    \end{equation*}
    This implies the existence of a $y_i \in (-2\s r, -\s r)$ such that
    \begin{equation}\label{cylinder estimate for lower bound}
        \h\big(\E_k(i) \cap (D_\p(0)\times\{y_i\})\big)\geq \h\Big(D_\p(0)\Big)\Big(1 - \frac{\s}{(1-4\s^2)^\frac{n}{2}}\Big).
    \end{equation}
    We may analogously find $y_j \in (\s r, 2\s r)$ such that \eqref{cylinder estimate for lower bound} holds for all $\E_k(j)$ with $k \geq k_0$. Next, we want to consider the cylinder $Q_{k,\s} = D_\p(0) \times[y_i,y_j]$ (we write $Q = Q_{k,\s}$ for brevity). As $Q$ is bounded, applying the divergence theorem to the constant vector field $-e_{n+1}$ and \cite{Maggi_Book} Theorem 16.3 shows that
    \begin{align}\label{cylinder divergence equation}
        0 = -\int_{Q\cap\E_k(i)} {\rm div}(e_{n+1})dx &= -\int_{Q\cap\pas\E_k(i)}e_{n+1}\cdot \nu_{\E_k(i)} \,d\h \\ 
        &-\int_{\{\nu_Q = \nu_{\E_k(i)}\}} e_{n+1}\cdot\nu_Q\,d\h\nonumber\\
        &-\int_{\E_k(i)^{(1)}\cap\pas Q} e_{n+1}\cdot\nu_Q\,d\h.\nonumber
    \end{align}
    We estimate each term in \eqref{cylinder divergence equation} separately.
    \begin{equation}\label{cylinder divergence estimate 1}
        -\int_{Q\cap\pas\E_k(i)} e_{n+1}\cdot\nu_{\E_k(i)}\,d\h \geq - \h\big(\pas \E_k(i) \cap Q\big).
    \end{equation}
    Next, we have by \eqref{cylinder estimate for lower bound} that
    \begin{align}\label{cylinder divergence estimate 2}
        &-\int_{\{\nu_Q = \nu_{\E_k(i)}\}} e_{n+1}\cdot\nu_Q\,d\h \nonumber\\ &= \h\big(\{\nu_{\E_k(i)} = -e_{n+1}\} \cap (D_\p(0)\times\{y_i\})\big) - \h\big(\{\nu_{\E_k(i)} = e_{n+1}\} \cap (D_\p(0)\times\{y_j\})\big)\nonumber\\
        &\geq -\h\big(\pas\E_k(i) \cap (D_\p(0)\times\{y_j\})\big) \geq -\h\big((D_\p(0)\times\{y_j\})\setminus \E_k(j)^{(1)}\big)\nonumber\\&\geq-\frac{\s}{(1-4\s^2)^\frac{n}{2}}\h\big(D_\p(0)\big),
    \end{align}
    where we have used that $\pas\E_k(i)\cap\E_k(j)^{(1)} = \emptyset$ and \eqref{cylinder estimate for lower bound} for $\E_k(j)$.
    Lastly, reusing the logic for \eqref{cylinder divergence estimate 2}, we conclude
    \begin{align}\label{cylinder divergence estimate 3}
        &-\int_{\E_k(i)^{(1)}\cap\pas Q} e_{n+1}\cdot\nu_Q \,d\h= \h\big(\E_k(i)\cap (D_\p(0)\times\{y_i\})\big) - \h\big(\E_k(i)\cap (D_\p(0)\times\{y_j\})\big) \nonumber\\
        &\geq \Big(1 - \frac{\s}{(1-4\s^2)^\frac{n}{2}}\Big)\h\big(D_\p(0)\big) - \h\big((D_\p(0)\times\{y_j\})\setminus \E_k(j)^{(1)}\big)\nonumber\\
        &\geq \Big(1 - \frac{2\s}{(1-4\s^2)^\frac{n}{2}}\Big)\h\big(D_\p(0)\big).
    \end{align}
    Plugging \eqref{cylinder divergence estimate 1}, \eqref{cylinder divergence estimate 2} and \eqref{cylinder divergence estimate 3} into \eqref{cylinder divergence equation}, we conclude that
    \begin{equation*}
        \h\big(\pas\E_k(i) \cap Q\big) \geq \Big(1 - \frac{3\s}{(1-4\s^2)^\frac{n}{2}}\Big)\h\Big(D_\p(0)\Big)
    \end{equation*}
    As the identical estimate holds on $\E_k(j)$, because $\E_k$ is a wet $N$-cluster we see that
    \begin{equation}\label{minimum perimeter in limit}
        \mu_k\big(B_r(0)\big)\ge\h\big(\pas\E_k(N+1) \cap Q\big) \geq \Big(2 - \frac{6\s}{(1-4\s^2)^\frac{n}{2}}\Big)\Big(1-4\s^2\Big)^\frac{n}{2}\h\Big(D_r(0)\Big)
    \end{equation}
    As $\mu$ is a finite measure, $\mu\big(\overline{B_r(0)}\big) = \mu\big(B_r(0)\big)$ for almost every $r.$ Thus \eqref{minimum perimeter in limit} we may find a sequence $r_h$ such that 
    \begin{equation}\label{lower bound on upper density sequence}
        \frac{\mu\big(B_{r_h}(0)\big)}{\h\big(D_{r_h}(0)\big)} \geq \limsup_{k\to\infty}\frac{\mu_k(B_{r_h}(0))}{\h\big(D_{r_h}(0)\big)}\geq \Big(2 - \frac{6\s}{(1-4\s^2)^\frac{n}{2}}\Big)\Big(1-4\s^2\Big)^\frac{n}{2}
    \end{equation}
    Sending $h\to\infty$, and then $\sigma\to 0$ in that order we have that
    \begin{equation}\label{upper density at least two}
        \t^*_n(\mu)(0) \geq 2.
    \end{equation}
    This is the desired result for this part of the proof.
      
    \medskip
  
    \noindent {\it Step four}: We now show that there is a compact $\h$-rectifiable set $K = {\rm supp}(\mu)$ such that $\mu = \theta \h\llcorner K.$ This step is almost identical to step four from \cite{King_2022}. We begin by picking $x\in K$ defining the sequence of functions
    \begin{equation*}
        f_k(r) = \mu_k\big(B_r(x)\big),\,\,\,\,\,\,\,\,\,\,\,\,\,\,\,\,\,\,\,\, f(r) = \mu\big(B_r(x)\big).
    \end{equation*}
    If $f'_k$ is the classical derivative of $f_k$ and $Df_k$ is the distributional one, by \cite{delellis2014directapproachplateausproblem}
    \begin{gather}\label{density function properties}
        f_k\to f \text{ for a.e. } r,\quad Df_k\geq f_k',\quad Df\geq f',\quad f'\geq\liminf_{k\to\infty} f_k',\\
        f_k'(r)\ge\H^{n-1}\big(\pa B_r(x)\cap\pas\E_k(N+1)\big).\nonumber
    \end{gather}
    In order to make full use of these properties, we want to introduce the cup and cone competitors to get estimates on $f.$ Up to modifying our subsequence, we can assume that
    \begin{equation*}
        P\big(\E_k(N+1)\big) \leq \psi_\e + \frac{1}{k}
    \end{equation*}
    Let $\F^k$ denote either the cone competitor $\F^k_{x,r}$ or a cup competitor $\F^k_{\eta,x,r}$ to $\E_k$ in $B_r(x)$, with $r < r_*$, where $r_*$ is a parameter which can be shrunk to ensure $d(\F^k,\E)$ is small enough that we may apply a volume fixing variation for large $k$ (see \cite{Maggi_Book} Theorem 29.14). In this circumstance, it holds that
    \begin{equation}\label{volume fixing variation error}
        P\big(\E_k(N+1)\big) \leq P\big(\F^k(N+1)\big) + C_*\sum_{h=1}^{N+1}\big||\E_k(h)| - |\F^k(h)|\big|  + \frac{1}{k},
    \end{equation}
    where $C_*$ is uniform in $k$.

    We first use the cone competitor to obtain an upper bound on $\mu$. Consider the cone competitor $\F_{x,r}^k$ where $r$ is chosen so that each $\F_{x,r}^k$ is a wet $N$-cluster. Then \eqref{cone competitor estimate} and \eqref{volume fixing variation error} will immediately imply that 
    \begin{equation*}
        f_k(r) \leq \frac{r}{n}\H^{n-1}\big(\pas\F^k(N+1) \cap \pa B_r(x)\big) + C(N,n)r^{n+1} +\frac{1}{k}.
    \end{equation*}
    Taking $k\to\infty$, we will see that
    \begin{equation}\label{post liminf fr estimate for cone}
        f(r)\leq \frac{r}{n}f'(r) + Cr^{n+1}.
    \end{equation}
    In particular $\big(f(r)/r^n\big)' \geq -nC$ for a.e. $r<r_*$, so that $r\mapsto f(r)/r^n + nCr$ is non-decreasing and
    \begin{equation}\label{crude upper density bound}
        f(r) \leq \Big(\frac{f(r_*)}{r_*^n} + nCr_*\Big)r^n \qquad \text{for all } r<r_*.
    \end{equation}
    In particular $\mu(\{x\}) = 0$ for every $x\in\rn$.

    We next use the cup competitors to obtain a lower bound. As we can assume without loss of generality that $\F^k_{\eta}$ is a valid cup competitor for all $k$, we have by Lemma \ref{cup competitor main} that \eqref{volume fixing variation error} implies
    \begin{equation}\label{initial estimate on fr}
        f_k(r) \leq 2\h\big(\pa B_r(x)\setminus A_k\big) + C_*\sum_{h=1}^{N+1}\big||\E_k(h)| - |\F_\eta^k(h)|\big|  + \frac{1}{k}.
    \end{equation}
    We now estimate the error term. Up to further shrinking $r_*$ we may assume $\frac{C_*r_*}{n+1} < \frac{1}{2}$.\\
    If $\h(A_k \cap \E_k(N+1)) = 0$, then
        \begin{align*}
            C_*\sum_{h=1}^{N+1}\big||\E_k(h)| -|\F_\eta^k(h)|\big| &\leq C_*\big(2|B_r|+ \big||\E_k(N+1)| -|\F_\eta^k(N+1)|\big|\big)\\
            &\leq C_*\big(2|B_r| + r|B_1|^{\frac{1}{n+1}}|\E_k(N+1)\cap B_r(x)|^{\frac{n}{n+1}} + \om_nr^n\eta\big)\\
            &\leq C_*\big(2|B_r| + \frac{r_*}{n+1}P\big(\E_k(N+1)\cap B_r(x)\big)+ \om_nr^n\eta\big)\\
            &\leq \frac{1}{2}\Big(f_k(r) + \h\big(\pa B_r(x)\setminus A_k\big)\Big) +C_*\Big(2|B_r| + \om_nr^n\eta\Big).
        \end{align*}
    If instead $A_k\hnsub\E_k(N+1),$ then
        \begin{align*}
            C_*\sum_{h=1}^{N+1}\big||\E_k(h)| - |\F_\eta^k(h)|\big| &\leq C_*\big(2|B_r|+ \big||B_r(x)\setminus\E_k(N+1)| -|B_r(x)\setminus\F_\eta^k(N+1)|\big|\big)\\
            &\leq C_*\big(2|B_r| + r|B_1|^{\frac{1}{n+1}}|B_r(x)\setminus\E_k(N+1)|^{\frac{n}{n+1}} + \om_nr^n\eta\big)\\
            &\leq C_*\big(2|B_r| + \frac{r_*}{n+1}P\big(B_r(x)\setminus\E_k(N+1)\big)+ \om_nr^n\eta\big)\\
            &\leq \frac{1}{2}\Big(f_k(r) + \h\big(\pa B_r(x)\setminus A_k\big)\Big) + C_*\Big(2|B_r| + \om_nr^n\eta\Big).
        \end{align*}
    In either case, taking $\eta\to0$ will give that 
    \begin{equation}\label{cup width goes to 0}
        C_*\sum_{h=1}^{N+1}\big||\E_k(h)| - |\F_\eta^k(h)|\big| \leq \frac{1}{2}\Big(f_k(r) + \h\big(\pa B_r(x)\setminus A_k\big)\Big) + c(n)r^{n+1}.
    \end{equation}
    Applying \eqref{cup width goes to 0} to \eqref{initial estimate on fr} yields
    \begin{equation}\label{pre liminf cup estimate}
        f_k(r) \leq 5\h(\pa B_r(x)\setminus A_k) + 2c(n)r^{n+1} + \frac{2}{k}.
    \end{equation}
    Assume first that $n\geq 2$. Applying Lemma \ref{spherical isoperimetry}, we see that
    \begin{equation}\label{spherical isoperimetry on the cup competitor}
        \h(\pa B_r(x)\setminus A_k) \leq  C(n)\H^{n-1}\big(\pa B_r\cap\pas\E_k(N+1)\big)^{\frac{n}{n-1}}\leq C(n)f_k'(r)^{\frac{n}{n-1}}
    \end{equation}
    Then \eqref{density function properties}, \eqref{spherical isoperimetry on the cup competitor} and \eqref{pre liminf cup estimate} will imply that
    \begin{align*}
        f(r) = \lim_{k\to\infty} f_k(r) &\leq \liminf_{k\to\infty} 5C(n)f_k'(r)^{\frac{n}{n-1}} + 2c(n)r^{n+1} + \frac{2}{k} \\&\leq 5C(n)f'(r)^{\frac{n}{n-1}} + 2c(n)r^{n+1}.
    \end{align*}
    By a standard ODE argument, this is sufficient to conclude that there exists a $\t_0>0$ such that for any $x\in K$, and almost every $r<r_*$, it must hold that 
    \begin{equation*}
        \mu\big(B_r(x)\big) \geq \t_0\om_nr^{n}.
    \end{equation*}

    \noindent {\it The case $n=1$}: Lemma \ref{spherical isoperimetry} is not available when $n=1$, so we argue directly. For a.e. $r<r_*$ the set $J_k=\pa B_r(x)\cap\pas\E_k(N+1)$ is finite and $f_k'(r)\geq\H^0(J_k)$. Since $\pas\E_k(N+1)$ is $\H^1$-equivalent to a countable union of closed Jordan curves, a generic circle meets it in an even number of points, so for a.e. $r$ either $\H^0(J_k)\geq 2$, whence $f_k'(r)\geq 2$, or $J_k=\emptyset$. In the latter case $\pa B_r(x)$ is $\H^1$-contained in a single chamber $\E_k(h)$, and every other chamber meets $B_r(x)$ in a set enclosed by $\pas\E_k(N+1)\cap B_r(x)$; by the planar isoperimetric inequality,
    \begin{equation*}
        \sum_{h'\neq h}|\E_k(h')\cap B_r(x)|\leq \frac{f_k(r)^2}{4\pi}.
    \end{equation*}
    The cup competitor of case (i) with $Y=\emptyset$ (that is, $\F^k(h)=\E_k(h)\cup B_r(x)$ and $\F^k(h')=\E_k(h')\setminus B_r(x)$ for $h'\neq h$) is a wet $N$-cluster with $P(\F^k(N+1))=P(\E_k(N+1))-f_k(r)$, so \eqref{volume fixing variation error} gives
    \begin{equation*}
        f_k(r)\leq \frac{C_*N}{4\pi}f_k(r)^2+\frac1k .
    \end{equation*}
    Hence, for a.e. $r<r_*$, either $f_k'(r)\geq2$, or $f_k(r)\leq 2/k$ provided $f_k(r)\leq 2\pi/(C_*N)$. By \eqref{crude upper density bound} we may shrink $r_*$ so that $f(r)<2\pi/(C_*N)$ for $r<r_*$, and passing to the limit via \eqref{density function properties} we find that for a.e. $r<r_*$ either $f'(r)\geq2$ or $f(r)=0$. As $x\in K=\mathrm{supp}\,\mu$, the second alternative never occurs, and integrating gives $f(r)\geq 2r=\om_1 r$ for $r<r_*$. Thus the lower bound above holds with $\t_0=1$ when $n=1$.

    In all cases we thus conclude
    \begin{equation}\label{lower bound for Preiss}
        \mu\geq \t_0\h\llcorner K.
    \end{equation}
    Combining \eqref{density function properties}, \eqref{lower bound for Preiss}, and \eqref{post liminf fr estimate for cone} we see that
    \begin{align}\label{monotonicity computation}
        D(\frac{e^{\Lambda r}f(r)}{r^n}) &= \frac{ne^{\Lambda r}}{r^{n+1}}\Big(\frac{r}{n}Df + f(r)(\frac{\Lambda r}{n} -1)\Big)\geq \frac{ne^{\Lambda r}}{r^{n+1}}\Big(f(r) - Cr^{n+1} + f(r)(\frac{\Lambda r}{n} -1)\Big) \nonumber\\&=\frac{ne^{\Lambda r}}{r^{n}}\Big( -Cr^n + \frac{\Lambda f(r)}{n}\Big) \geq ne^{\Lambda r}\Big(\frac{\Lambda\t_0\om_n}{n} -C\Big).
    \end{align}
    In other words, if $\Lambda\geq \frac{nC}{\t_0\om_n}$ then $e^{\Lambda r}\frac{f(r)}{r^{n}}$ is non-decreasing. This alongside \eqref{lower bound for Preiss} implies that $\t(x)$ exists in $(0,\infty)$ for every $x\in K$. Thus Preiss' theorem implies that $\mu = \theta\h\llcorner K$ where $\t$ is upper semicontinuous and $K$ is countably rectifiable.
    
    \medskip
  
    \noindent {\it Step five}: Given an $x\in K$ such that $K$ admits an approximate tangent plane, we show the following estimates on $\t$. If $x\in\pas \E(N+1)$ then we show $\t(x) \leq 1$. If $x\in\E(N+1)^{(0)}$ or $x\in\E(N+1)^{(1)}$, then we show that $\t(x)\leq 2$. Choose $\nu$ such that $\nu^{\perp} = T_xK$, $\tau\in(0,1)$, $\sigma\in(0,\tau),$ and define the following sets:
    \begin{align}\label{slab sub-parts}
        S_{\tau,r} &= \{y\in B_r(x):|(y-x)\cdot\nu|<\tau r\}\\
        V_{\s,r} &= \{y\in B_r(x):|(y-x)\cdot \nu|<\s|y-x|\}\nonumber\\
        W_{\tau,\s,r}^\pm &=\{y\in B_r(x):(y-x)\cdot\nu \gtrless 0\} \cap \big(S_{\tau,r}\setminus\overline{V}_{\s,r}\big)\nonumber\\
        \Gamma^{\pm}_{\tau,\s,r} &= \pas S_{\tau,r} \cap \pas W_{\tau,\s,r}^\pm.\nonumber
    \end{align}
    We observe that \eqref{lower bound for Preiss}, $\h\llcorner \frac{K-x}{\rho} \weakstar \h\llcorner T_x K$ and the portmanteau theorem will imply the existence of an $r_0 = r_0(\sigma,x) >0$ such that
    \begin{equation}\label{classical tangent plan is in cone for slab}
        K\cap B_r(x) \subset V_{\s,r}\cup\{x\}.
    \end{equation}
    In particular, for $r<r_0$, it holds that
    \begin{equation}
        \mu(S_{\tau,r}) = \mu(B_r).
    \end{equation}
    Next we define $A_{r,k}^{\rm in}$ and $A_{r,k}^{\rm out}$ to be the $\h$-maximal elements of the sets
    \begin{align*}
        \A^{\rm in}_{r,k} &= \{A\subset \pa S_{\tau,r} : \h(A\setminus\E_k(N+1)) = 0, A \,\,{\rm indecomposable}\}\\
        \A^{\rm out}_{r,k} &= \{A\subset \pa S_{\tau,r} : \h(A\cap\E_k(N+1)) = 0, A \,\,{\rm indecomposable}\}.
    \end{align*}
    These give rise to the slab competitors $\F^*_k \coloneqq \mathcal{S}_{\eta,\tau,x,r,A_{r,k}^*}^\nu$ for $*\in\{{\rm in,out}\}$ as defined in Section \ref{subsec:slab}. Furthermore, \eqref{slab competitor main estimate 1} and \eqref{slab competitor main estimate 2} will imply that 
    \begin{equation}\label{slab in estimates}
        \limsup_{\eta\to0}\h(S_{\tau,r}\cap\pas\F^{\rm in}_k(N+1))\leq C(n,\tau)
            \h((\pa S_{\tau,r}\cap\E_k(N+1))\setminus A_{r,k}^{\rm in}),
    \end{equation}

    \begin{equation}\label{slab out estimates}
        \limsup_{\eta\to0}\h(S_{\tau,r}\cap\pas\F^{\rm out}_k(N+1))\leq C(n,\tau) \h(\pa S_{\tau,r}\setminus(\E_k(N+1)\cup A_{r,k}^{\rm out})).
    \end{equation}

    If $r_0$ is sufficiently small, we may apply \eqref{volume fixing variation error} to deduce 
    \begin{equation}\label{slab estimate with volume fixing}
        \h(S_{\tau,r}\cap\pas\E_k(N+1)) \leq \h(S_{\tau,r}\cap\F^*_{k}) + C_*c(n,N) + \frac{1}{k}, \,\, {\rm for}\,\, * = \{{\rm in,out}\},
    \end{equation}

   as $\h(\pa S_{\tau,r}\cap\pas\F_k) = 0$ for almost every $r$. Then combining \eqref{slab in estimates}, \eqref{slab out estimates} and \eqref{slab estimate with volume fixing}, we see that by first taking $\eta\to0$, and then taking $k\to\infty$, we find that for $* =$ in,
    \begin{align}\label{in estimate}
        \mu(B_r(x)) &\leq \limsup_{k\to\infty}\h(\pa S_{\tau,r}\setminus \E_k(N+1))\\
        &+ C(n,\tau)\limsup_{k\to\infty}\h((\pa S_{\tau,r}\cap \E_k(N+1))\setminus A^{\rm in}_{r,k}) + C_* c(n,N)r^{n+1}.\nonumber
    \end{align}
    Similarly, when $* =$ out, we see that
    \begin{align}\label{out estimate}
        \mu(B_r(x)) &\leq \limsup_{k\to\infty}\h(\pa S_{\tau,r}\cap \E_k(N+1))\\
        &+ C(n,\tau)\limsup_{k\to\infty}\h(\pa S_{\tau,r}\setminus (\E_k(N+1)\cup A^{\rm out}_{r,k})) + C_* c(n,N)r^{n+1}\nonumber.
    \end{align}
    We now discuss the cases $x\in\pas \E(N+1)$, $x\in\E(N+1)^{(0)}$ and $x\in\E(N+1)^{(1)}$ separately.

    \medskip

    \noindent \textit{The case} $x\in\pas\E(N+1)$: First we observe by the divergence theorem and \eqref{classical tangent plan is in cone for slab} that \begin{equation}\label{measures on the Wpm sets}
        |W^+_{\tau,\s,r_0}\cap \E(N+1)| = 0.
    \end{equation}
    Thus we may apply the dominated convergence theorem, coarea formula, and Fatou's lemma in that order to \eqref{measures on the Wpm sets} to see that
    \begin{align}
        0 = \lim_{k\to\infty}|W^+_{\tau,\s,r_0}\cap \E_k(N+1)| &= \lim_{k\to\infty}\int_0^{r_0}\h(\pa S_{\tau,r}\cap W^+_{\tau,\s,r_0}\cap \E_k(N+1))\,dr\\
        &\geq \int_0^{r_0}\liminf_{k\to\infty} \h(\Gamma^+_{\tau,\sigma,r}\cap \E_k(N+1)).\nonumber
    \end{align}
    This along with a similar computation for $W^-_{\tau,\s,r}\setminus\E(N+1)$, allow us to conclude that for almost every $r<r_0$, 
    \begin{align}
        &\lim_{k\to\infty}\h(\Gamma^+_{\tau,\sigma,r}\cap \E_k(N+1)) = 0\label{estimate on gamma plus}\\
        &\lim_{k\to\infty}\h(\Gamma^-_{\tau,\sigma,r}\setminus \E_k(N+1)) = 0.\label{estimate on gamma minus} 
    \end{align}
    By \eqref{estimate on gamma minus}, \eqref{estimate on gamma plus}, and that 
    \begin{equation}\label{decomposition of slab boundary in terms of smaller slabs}
        \pa S_{\tau,r} = \Gamma^+_{\tau,\s,r}\cup\Gamma^-_{\tau,\s,r}\cup(\pa S_{\tau,r}\cap\pa S_{\s,r}),
    \end{equation}
    we have that 
    \begin{align}\label{circle slab estimate red bdry}
        |\h(\pa S_{\tau,r}\cap \E_k(N+1))-\om_nr^n|&\leq C(n)\s r^n + |\h(\Gamma^-_{\tau,\s,r}\cap\E_k(N+1))-\om_nr^n| + o(1)\nonumber\\
        &\leq C(n)\sigma r^n + |\h(\Gamma^-_{\tau,\s,r}) -\om_nr^n| + o(1)\nonumber\\
        &\leq C(n)\tau r^n + o(1).
    \end{align}
    Next, we assume that without loss of generality that $r_0 = r_0(\s,x)$ is chosen so that $\h(K\cap \pa B_{r_0}(x)) = 0$. Then by \eqref{classical tangent plan is in cone for slab}, we see that
    \begin{align*}
        0 = \mu(K\cap \overline{B_{r_0}}(x)\setminus V_{\s,r_0}) &\geq \lim_{k\to\infty}\h(S_{\tau,r_0}(x)\cap\pas\E_k(N+1)\setminus V_{\s,r_0})\\
        &=\lim_{k\to\infty}\int_0^{r_0}\H^{n-1}(\pa S_{\tau,r}\cap\pas\E_k(N+1)\setminus V_{\s,r_0})dr.
    \end{align*}
    Thus for almost any $r < r_0$, we see that
    \begin{equation*}
        \lim_{k\to\infty}\H^{n-1}(\pa S_{\tau,r}\cap\pas\E_k(N+1)\setminus V_{\s,r}) = 0,
    \end{equation*}
    which implies that 
    \begin{equation}\label{gamma plus red bdry vanishes}
        \lim_{k\to\infty}\H^{n-1}(\Gamma^+_{\tau,\s,r}\cap \pas\E_k(N+1)) = 0.
    \end{equation}
    Using Lemma \ref{spherical isoperimetry} and the fact that $\Gamma^+_{\tau,\s,r}$ is a bi-Lipschitz image of a spherical cap, we may further deduce that if
    \begin{equation*}
        A^+_{r,k} \text{ is an $\h$-maximal indecomposable component of } \Gamma^+_{\tau,\s,r}\setminus\E_k(N+1),
    \end{equation*}
    then we have that
    \begin{equation}\label{spherical isoperimetry in reduced bdry slab}
        \lim_{k\to\infty}\h(\Gamma^+_{\tau,\s,r}\setminus A^+_{r,k}) = 0.
    \end{equation}
    By the $\h$-maximality of $A^{\rm out}_{r,k}$ and \eqref{spherical isoperimetry in reduced bdry slab}, it must eventually hold that up to sets of $\h$-measure 0, $A^+_{r,k} \subset A^{\rm out}_{r,k}$. Thus we see that
    \begin{equation}\label{improved red bdry slab}
        \lim_{k\to\infty}\h(\Gamma^+_{\tau,\s,r}\setminus A^{\rm out}_{r,k}) = 0.
    \end{equation}
    Combining \eqref{estimate on gamma minus}, \eqref{decomposition of slab boundary in terms of smaller slabs} and \eqref{improved red bdry slab}, we see that
    \begin{equation*}
        \limsup_{k\to\infty}\h\big(\pa S_{\tau,r}\setminus (A^{\rm out}_{r,k}\cup\pas\E_k(N+1))\big) \leq \h(\pa S_{\tau,r}\cap\pa S_{\s,r}) \leq C(n)\s r^n.
    \end{equation*}
    This in conjunction with \eqref{out estimate} and \eqref{circle slab estimate red bdry} shows that 
    \begin{equation}\label{reduced boundary final density estimate}
        \frac{\mu(B_r(x))}{r^n}\leq \om_n +  C(n)\tau + C(n,\tau)\s + C_*c(n,N)r.
    \end{equation}
    Taking $r\to0, \s\to0$ and then $\tau\to0$ will imply that $\t(\mu)(x) \leq 1$.

    \medskip

    \noindent \textit{The case} $x\in\E(N+1)^{(0)}$:
    By \eqref{classical tangent plan is in cone for slab}, the divergence theorem and that $\pas\E(N+1)\subset K$ we see that 
    \begin{equation*}
        \H^{n+1}(\E(N+1)\cap B_{r_0}(x)\setminus V_{\s,r_0}) = 0.
    \end{equation*}
    By the dominated convergence, coarea formula, Fatou's lemma and the above, we see that
    \begin{equation*}
        0 = \lim_{k\to\infty}\h\big((\E_k(N+1)\setminus V_{\s,r})\cap \pa S_{\tau,r}\big) = \lim_{k\to\infty}\h\big(\E_k(N+1)\cap(\Gamma_{\tau,\s,r}^+\cup \Gamma_{\tau,\s,r}^-)\big).
    \end{equation*}
    By \eqref{decomposition of slab boundary in terms of smaller slabs} and the above we see that
    \begin{equation}\label{zero density slab boundary estimate}
        \h(\E_k(N+1)\cap\pa S_{\tau,r}) = \h(\pa S_{\tau,r} \cap \pa S_{\s,r}) + o(1) \leq C(n)\s r^n + o(1).
    \end{equation}
    Similarly, 
    \begin{align}\label{circle slab estimate zero density}
        |\h(\pa S_{\tau,r}\setminus \E_k(N+1)) -2\om_nr^n| &\leq  \h(\pa S_{\tau,r} \cap \pa S_{\s,r}) + |\h(\Gamma^+_{\tau,\s,r}\cup\Gamma_{\tau,\s,r}^-) - 2\om_nr^n|
        \nonumber\\ &\quad + o(1) \leq C(n)\tau r^n + o(1).
    \end{align}
    Combining \eqref{in estimate}, \eqref{zero density slab boundary estimate} and \eqref{circle slab estimate zero density} in the same manner as how \eqref{reduced boundary final density estimate} was derived, we conclude $\theta(\mu)(x) \leq 2$.
    
    \medskip

    \noindent \textit{The case} $x\in\E(N+1)^{(1)}$: Similar to the density 0 case, we have that 
    \begin{equation*}
        |B_{r_0}(x)\setminus(\E(N+1) \cup V_{\s,r})| = 0.
    \end{equation*}
    Again by dominated convergence, the coarea formula, Fatou's lemma and the above we see that 
    \begin{equation*}
        \lim_{k\to\infty}\h((\Gamma^+_{\tau,\s,r}\cup\Gamma^-_{\tau,\s,r})\setminus\E_k(N+1)) = 0.
    \end{equation*}
    Applying the identical logic of the density 0 case to \eqref{out estimate}, we again conclude that $\t(\mu)(x) \leq 2$.
    
    \medskip
  
    \noindent {\it Step six}: By Proposition 12.19 of \cite{Maggi_Book} (see Remark \ref{normalization remark} below), we may modify a cluster $\E$ on sets of measure zero to ensure that $\overline{\pas\E(i)} = \pa\E(i)$ for all $i$. In this setting, we show that if $x\in K\cap\E(h)^{(1)}$ for some $h\in\{0,\dots,N+1\}$ and $K$ admits an approximate tangent plane at $x$, then $\t(\mu)(x)=0$. Since $K$ is countably $\h$-rectifiable and, by Federer's theorem, $\h$-a.e. $x\in\rn\setminus\bigcup_{h}\pas\E(h)$ belongs to $\E(h)^{(1)}$ for exactly one $h\in\{0,\dots,N+1\}$, this shows that $\mu$ is concentrated on $\bigcup_h\pas\E(h)$. To do so we expand on step five by defining $B_{r,k}^{\rm in}$ and $B_{r,k}^{\rm out}$ to be $\h$-maximal components of $\A_{r,k}^{\rm in}\setminus\{A_{r,k}^{\rm in}\}$ and $\A_{r,k}^{\rm out}\setminus\{A_{r,k}^{\rm out}\}$ respectively. Then by Remark \ref{multi component wet slab}, we may improve \eqref{in estimate} to
    \begin{align}\label{multi in estimate}
        \mu(B_r(x)) &\leq \limsup_{k\to\infty}\h(\pa S_{\tau,r}\setminus \E_k(N+1))\\
        &+ C(n,\tau)\limsup_{k\to\infty}\h((\pa S_{\tau,r}\cap \E_k(N+1))\setminus (A^{\rm in}_{r,k}\cup B^{\rm in}_{r,k})) + C_* c(n,N)r^{n+1}\nonumber.
    \end{align}
    The other case is a bit more delicate, but if there is an $h$ such that, for $k$ large,
    \begin{equation}\label{containment condition for bout}
        \h((A_{r,k}^{\rm out} \cup B_{r,k}^{\rm out})\setminus \E_k(h)) = 0,
    \end{equation} then we may improve \eqref{out estimate} to
    \begin{align}\label{multi out estimate}
        \mu(B_r(x)) &\leq \limsup_{k\to\infty}\h(\pa S_{\tau,r}\cap \E_k(N+1))\\
        &+ C(n,\tau)\limsup_{k\to\infty}\h(\pa S_{\tau,r}\setminus (\E_k(N+1)\cup A^{\rm out}_{r,k} \cup B_{r,k}^{\rm out})) + C_* c(n,N)r^{n+1}\nonumber.
    \end{align}
    As in step five, there is casework depending on which chamber $x$ lies in.

    \medskip

    \noindent \textit{A preliminary estimate}: Let $x\in K\cap\E(h)^{(1)}$ for some $h\in\{0,\dots,N+1\}$, and let $\nu$, $\tau$, $\s$, $S_{\tau,r}$, $V_{\s,r}$ and $r_0=r_0(\s,x)$ be as in step five, where we now also assume $\tau<1/2$. We claim that, up to decreasing $r_0$,
    \begin{equation}\label{density one volume estimate}
        |B_r(x)\setminus\E(h)|\le C(n)\,\s\,r^{n+1}\qquad\text{for every } r<r_0.
    \end{equation}
    Indeed, let $U^\pm_r=\{y\in B_r(x):\pm(y-x)\cdot\nu>\s|y-x|\}$, so that $B_r(x)\setminus(U^+_r\cup U^-_r)=\{y\in B_r(x):|(y-x)\cdot\nu|\le\s|y-x|\}$ has volume at most $C(n)\s r^{n+1}$, while $|U^\pm_r|\ge c(n)r^{n+1}$ as $\s<1/2$. By \eqref{classical tangent plan is in cone for slab}, $U^\pm_r\cap K=\emptyset$ for $r<r_0$. By step two (whose argument applies verbatim to $h'=0$) and the normalization $\pa\E(h')=\overline{\pas\E(h')}$, we have $\pa\E(h')={\rm spt}\,\mu_{\E(h')}\subset K$ for every $h'=0,\dots,N+1$. Hence $\mu_{\E(h')}\llcorner U^\pm_r=0$, so that $1_{\E(h')}$ is a.e. constant on each of the connected open sets $U^\pm_r$ (see \cite[Lemma 7.5]{Maggi_Book}); since the chambers partition $\rn$ up to null sets, each $U^\pm_r$ is Lebesgue-a.e. contained in a single chamber. As $x\in\E(h)^{(1)}$, we have $|B_r(x)\setminus\E(h)|=o(r^{n+1})<c(n)r^{n+1}\le|U^\pm_r|$ for $r$ small, which forces $|U^\pm_r\setminus\E(h)|=0$ for $r<r_0$, and \eqref{density one volume estimate} follows. Next, let $g(y)=\max\{|y-x|,|(y-x)\cdot\nu|/\tau\}$, so that $\{g<r\}=S_{\tau,r}$ and $\{g=r\}=\pa S_{\tau,r}$. As the maximum of two Lipschitz functions with constants $1$ and $1/\tau$, $g$ is Lipschitz with $\Lip(g)\le1/\tau$, so that $|\nabla g|\le1/\tau$ a.e. and the coarea formula for Lipschitz functions gives, for every measurable $F\subset\rn$ and $\rho<r_0$,
    \begin{equation}\label{slab coarea}
        \int_0^{\rho}\h(\pa S_{\tau,r}\cap F)\,dr=\int_{F\cap S_{\tau,\rho}}|\nabla g|\,dy\le\frac1\tau\,|F\cap S_{\tau,\rho}|.
    \end{equation}
    Applying \eqref{slab coarea} with $F=\rn\setminus\E_k(h)$, restricting the integral to $(\rho/2,\rho)$, and then using Fatou's lemma on the left and $\E_k\to\E$, $S_{\tau,\rho}\subset B_\rho(x)$ and \eqref{density one volume estimate} on the right, we find, with $f(r):=\liminf_{k\to\infty}\h(\pa S_{\tau,r}\setminus\E_k(h))$,
    \begin{equation}\label{slab integral bound}
        \int_{\rho/2}^{\rho}f(r)\,dr\le\liminf_{k\to\infty}\frac{1}{\tau}\,|S_{\tau,\rho}\setminus\E_k(h)|=\frac{1}{\tau}\,|S_{\tau,\rho}\setminus\E(h)|\le\frac{C(n)}{\tau}\,\s\,\rho^{n+1}.
    \end{equation}
    Setting $M=4C(n)\tau^{-1}\s\rho^n$, Chebyshev's inequality and \eqref{slab integral bound} give $|\{r\in(\rho/2,\rho):f(r)>M\}|\le M^{-1}\int_{\rho/2}^\rho f\,dr\le\rho/4$, so that $f\le M$ on a subset of $(\rho/2,\rho)$ of measure at least $\rho/4$. Since $\rho<2r$ for $r\in(\rho/2,\rho)$, we have $M\le 2^{n+2}C(n)\tau^{-1}\s\,r^n=:C(n,\tau)\,\s\,r^n$ there. Hence the set $G_\s$ of those $r<r_0$ satisfying the almost-everywhere conditions of step five and
    \begin{equation}\label{density one slab boundary estimate}
        \liminf_{k\to\infty}\h(\pa S_{\tau,r}\setminus\E_k(h))\le C(n,\tau)\,\s\,r^n
    \end{equation}
    has $|G_\s\cap(\rho/2,\rho)|\ge\rho/4$ for every $\rho<r_0$; in particular $0$ is an accumulation point of $G_\s$. In the two cases below we fix $r\in G_\s$ and, passing to a subsequence in $k$ (along which all the estimates of step five remain valid), we assume that the $\liminf$ in \eqref{density one slab boundary estimate} is a limit. Since $\t(\mu)(x)$ exists by step four, in order to prove $\t(\mu)(x)=0$ it suffices to show
    \begin{equation}\label{density one target estimate}
        \mu(B_r(x))\le C(n)\tau r^n + C(n,\tau)\,\s\,r^n+C_*c(n,N)r^{n+1}\qquad\text{for every } r\in G_\s,
    \end{equation}
    and then let $r\to0$ in $G_\s$, $\s\to0$ and finally $\tau\to0$.

    \medskip

    \noindent \textit{The case} $x\in\E(N+1)^{(1)}$:
    For $r\in G_\s$, \eqref{density one slab boundary estimate} reads
    \begin{equation}\label{interior slab limit estimate}
        \lim_{k\to\infty}\h(\pa S_{\tau,r} \setminus\E_{k}(N+1)) \le C(n,\tau)\,\s\,r^n,
    \end{equation}
    which implies
    \begin{equation}\label{gamma minus liquid vanishes}
        \limsup_{k\to\infty}\h(\Gamma^+_{\tau,\s,r}\setminus\E_k(N+1)) + \limsup_{k\to\infty}\h(\Gamma^-_{\tau,\s,r}\setminus\E_k(N+1)) \le C(n,\tau)\,\s\,r^n.
    \end{equation}
    As it still holds for almost every $r<r_0$ that $\mu(K \cap \overline{B_{r_0}}(x)\setminus V_{\s,r_0})=0$, we may apply the identical process that derived \eqref{gamma plus red bdry vanishes} to $\Gamma^+_{\tau,\s,r}$ and $\Gamma^-_{\tau,\s,r}$ to conclude that
    \begin{equation}\label{gamma pm red bdry vanishes}
        \lim_{k\to\infty}\H^{n-1}(\Gamma^+_{\tau,\s,r}\cap\pas\E_k(N+1)) = \lim_{k\to\infty}\H^{n-1}(\Gamma^-_{\tau,\s,r}\cap\pas\E_k(N+1)) = 0.
    \end{equation}
    By applying Lemma \ref{spherical isoperimetry} to the above, we may further conclude that if 
    \begin{align*}
        &A^+_{r,k} \text{ is an $\h$-maximal indecomposable component of } \Gamma^+_{\tau,\s,r}\setminus\pas\E_k(N+1)\\
        &A^-_{r,k} \text{ is an $\h$-maximal indecomposable component of } \Gamma^-_{\tau,\s,r}\setminus\pas\E_k(N+1),
    \end{align*}
    then
    \begin{align}\label{Gamma plus is overwhelmed}
        \lim_{k\to\infty}\h(\Gamma^+_{\tau,\s,r}\setminus A^+_{r,k}) = 0.\\\label{Gamma minus is overwhelmed} \lim_{k\to\infty}\h(\Gamma^-_{\tau,\s,r}\setminus A^{-}_{r,k}) = 0.
    \end{align}
    Since $A^\pm_{r,k}$ is indecomposable and disjoint from $\pas\E_k(N+1)\supset\bigcup_{h'}\pas\E_k(h')$, it is $\h$-contained in a single chamber of $\E_k$. Moreover, $\Gamma^\pm_{\tau,\s,r}$ contains the flat face $\{(y-x)\cdot\nu=\pm\tau r\}\cap B_r(x)$ (recall $\s<\tau$), so that $\h(\Gamma^\pm_{\tau,\s,r})\ge\om_n(1-\tau^2)^{n/2}r^n\ge c(n)r^n$, and \eqref{Gamma plus is overwhelmed}, \eqref{Gamma minus is overwhelmed} give $\h(A^\pm_{r,k})\ge c(n)r^n/2$ for $k$ large. If $\s<\s_0(n,\tau)$, this is incompatible with \eqref{interior slab limit estimate} unless $A^\pm_{r,k}$ is $\h$-contained in $\E_k(N+1)$. Hence, for large enough $k$,
    \begin{equation}\label{the good containment}
        A^+_{r,k}, A^-_{r,k} \subset \E_k(N+1).
    \end{equation}
    Observe that by choosing our initial $\s$ small enough, we may ensure 
    \begin{equation}\label{gammas are big}
        \h(\pa S_{\tau,r} \setminus (\Gamma^+_{\tau,\s,r} \cup \Gamma^-_{\tau,\s,r})) \leq \frac{1}{2}\h(\Gamma^+_{\tau,\s,r}).
    \end{equation}
    Thus we assume without loss of generality, $\s$ starts small enough such that \eqref{gammas are big} holds. Then, \eqref{Gamma plus is overwhelmed} and \eqref{Gamma minus is overwhelmed} together imply for large enough $k$,
    \begin{equation}\label{A plus minus is big too}
        \h(\pa S_{\tau,r} \setminus(A_{r,k}^+ \cup A_{r,k}^-)) \leq \min\{\h(A^+_{r,k}),\h(A^-_{r,k})\}.
    \end{equation}
    The $\h$-maximality of $A^{\rm in}_{r,k}$ and $B^{\rm in}_{r,k}$, \eqref{gamma minus liquid vanishes}, \eqref{the good containment} and \eqref{A plus minus is big too} imply for large enough $k$ 
    \begin{equation}\label{A pm inside A in B in}
        A^+_{r,k} \cup A_{r,k}^- \hnsub A_{r,k}^{\rm in} \cup B^{\rm in}_{r,k}.
    \end{equation}
    Thus it must hold that
    \begin{equation}\label{interior multi component slab in estimate}
        \lim_{k\to\infty}\h(\Gamma^{+}_{\tau,\s,r}\cup\Gamma^-_{\tau,\s,r}\setminus(A^{\rm in}_{r,k}\cup B^{\rm in}_{r,k})) = 0.
    \end{equation}
    Combining \eqref{decomposition of slab boundary in terms of smaller slabs}, \eqref{multi in estimate}, \eqref{interior slab limit estimate} and \eqref{interior multi component slab in estimate} we see that
    \begin{equation*}
        \mu(B_r(x))\leq C(n)\tau r^n + C(n,\tau)\s r^n + C_*c(n,N)r^{n+1}.
    \end{equation*}
    This is \eqref{density one target estimate}, so that $\theta(\mu)(x)=0$.

    \medskip

    \noindent \textit{The case} $x\in\E(h)^{(1)}$ \text{ for } $h\neq N+1$:
    In this case we take $r\in G_\s$, with $G_\s$ defined through \eqref{density one slab boundary estimate} for the chamber $\E(h)$.
    We may follow the previous case with only two minor modifications. The first is that we replace $\E(N+1)$ with $\E(h)$ in equations \eqref{interior slab limit estimate}, \eqref{gamma minus liquid vanishes} and \eqref{the good containment}. These equations via the verbatim reasoning become
    \begin{align}\label{new1}
        \lim_{k\to\infty}&\h(\pa S_{\tau,r} \setminus\E_{k}(h)) \le C(n,\tau)\,\s\,r^n,\\\nonumber
        &\limsup_{k\to\infty}\h(\Gamma^+_{\tau,\s,r}\setminus\E_k(h)) + \limsup_{k\to\infty}\h(\Gamma^-_{\tau,\s,r}\setminus\E_k(h)) \le C(n,\tau)\,\s\,r^n,\\\label{new3}
        &A^+_{r,k}, A^{-}_{r,k} \subset \E_k(h).
    \end{align}
    Observe that \eqref{new1} implies that
    \begin{equation}\label{first part of out estimate for interior}
        \lim_{k\to\infty} \h(\pa S_{\tau,r}\cap\E_k(N+1)) \le C(n,\tau)\,\s\,r^n.
    \end{equation}
    The second modification is that we must consider the set $\A^{\rm out}_{r,k}$. The same considerations taken in the previous case will allow us to conclude the analogues of \eqref{A pm inside A in B in} and \eqref{interior multi component slab in estimate}:
    \begin{align}\nonumber
        &A^+_{r,k} \cup A_{r,k}^- \hnsub A_{r,k}^{\rm out} \cup B^{\rm out}_{r,k},\\\label{interior multi slab out estimate}
         \lim_{k\to\infty}\h(&\Gamma^{+}_{\tau,\s,r}\cup\Gamma^-_{\tau,\s,r}\setminus(A^{\rm out}_{r,k}\cup B^{\rm out}_{r,k})) = 0.
    \end{align}
    Combining \eqref{decomposition of slab boundary in terms of smaller slabs}, \eqref{containment condition for bout} (which holds by \eqref{new3}, since $A^{\rm out}_{r,k}$ and $B^{\rm out}_{r,k}$ are indecomposable, disjoint from $\pas\E_k(N+1)$, and contain $A^+_{r,k}$ and $A^-_{r,k}$), \eqref{multi out estimate}, \eqref{new3}, \eqref{first part of out estimate for interior} and \eqref{interior multi slab out estimate} we obtain the identical density estimate
    \begin{equation*}
        \mu(B_r(x))\leq C(n)\tau r^n + C(n,\tau)\s r^n + C_*c(n,N)r^{n+1}.
    \end{equation*}
    This is again \eqref{density one target estimate}, so that $\theta(\mu)(x)=0$.
    
    \medskip
  
    \noindent {\it Step Seven:} We now may complete the theorem. Steps four and six imply that $\mu$ is concentrated on the set $$\bigcup_{h=1}^{N+1}\pas \E(h).$$ Step three, step five, and the above imply that $$\mu = 2\h\llcorner(\bigcup_{h=0}^N \pas \E(h)\setminus\pas\E(N+1)) + \h\llcorner(\pas\E(N+1)).$$ 
    Immediately, we may conclude $\psi_\e(v_1,\dots, v_N) = P^*_{\rm wet}(\E),$ proving the theorem. 
    
\end{proof}
\begin{remark}\label{normalization remark}
    Given a cluster $\E$ satisfying $$\h(\bigcup_{h=1}^{N+1} \pa\E(h)) < \infty,$$ we may define a new cluster $\E^*$ by
    \begin{align*}
        \E^*(h) &= \big(A_1(h)\cup\E(h)\big)\setminus A_0(h),\\
        A_1(h) &= \{x : \exists r>0 \text{ s.t. } |B_r(x)\cap \E(h)| = |B_r(x)|\},\\
        A_0(h) &= \{x : \exists r>0 \text{ s.t. } |B_r(x)\cap \E(h)| = 0\},
    \end{align*}
    as in the proof of \cite[Proposition 12.19]{Maggi_Book},
    such that $|\E(h)\Delta\E^*(h)| = 0$ and $\pa\E^*(h) = {\rm spt}\,\mu_{\E^*(h)} = \overline{\pas\E^*(h)}$. In other words, by the previous theorem we may assume that for all future minimizing clusters $\E$ each chamber is a Borel set satisfying $\overline{\pas \E(h)} = \pa\E(h)$. (The open set $A_1(h)$ alone is in general not Lebesgue equivalent to $\E(h)$, as $\E(h)\cap{\rm spt}\,\mu_{\E(h)}$ may have positive measure.)
\end{remark}

\section{Equivalence of the relaxed problem}
This section of the paper will be entirely dedicated to proving Theorem \ref{relaxation thm main}, and introducing a few key lemmas which are essential to the proof. This section is a direct application of the methods used in \cite{Novack_2024}. First, we prove an intermediate theorem that states that any $(N+1)$-cluster may be approximated by a wet $N$-cluster in the following sense: 
\begin{theorem}\label{approximation thm main later}
    If $\E$ is an $(N+1)$-cluster, then given any $\de>0$, there exists a wet $N$-cluster $\F$ such that 
    \begin{align*}
        |\E(h)| &= |\F(h)|, \,\, h=1,\dots,N+1,\\
        d(\E,\F) &\le \de,\\
        P_{\rm wet}(\F) &\leq P_{\rm wet}^*(\E) + \de.
    \end{align*}
\end{theorem}
This theorem is almost identical to \cite{Novack_2024} Theorem 2.3, but since we are working in the setting of clusters instead of the setting established in \cite{King_2022} we make a few necessary modifications and re-prove the theorem. Theorem \ref{relaxation thm main} will follow almost immediately.
\begin{proof}
    [Proof of Theorem \ref{approximation thm main later}]
  
    \medskip
  
    \noindent {\it Step one}: We show it suffices to prove the following statement: Given an $(N+1)$-cluster $\E$, there exists a wet $N$-cluster $\E_j$ and a sequence $\de_j\to0$ such that for all $h$,
    \begin{align}
        d(\E_j,\E) &\le \de_j,\\
        P_{\rm wet}(\E_j) &\leq P_{\rm wet}^*(\E) + \de_j.
    \end{align}
    Indeed, if $\E_j$ is one such sequence, then by \cite{Maggi_Book} Section 29.6 we have the existence of a sequence of diffeomorphisms $\psi_j$ such that 
    \begin{align*}
        |\psi_j(\E_j(h))| &= |\E(h)|,\\  d(\psi_j(\E_j),\E) &\leq C_0\delta_j,\\P_{\rm wet}(\psi_j(\E_j)) \leq  P_{\rm wet}(\E_j)(1 + 2C_1d(\E_j,\E)) &\leq (P_{\rm wet}^*(\E) + \de_j)(2\de_jC_1 + 1).
    \end{align*}
    Thus for a small enough $\de_j$, we have the conditions 
    \begin{align*}
        |\psi_j(\E_j(h))| &= |\E(h)|,\\  d(\psi_j(\E_j),\E) &\leq \de,\\P_{\rm wet}(\psi_j(\E_j)) &\leq P_{\rm wet}^*(\E) + \de.
    \end{align*}
    These conditions are exactly the desired result of the theorem.
  
    \medskip
  
    \noindent {\it Step two}: Here, we will recall some of the important results from the corresponding proof in \cite{Novack_2024}, the proofs may be found there. We merely state the information here for the convenience of the reader.
    
    \medskip
  
    \noindent {\bf (i)}: Given an $\h$-rectifiable set $R$ with $\h(R)<\infty$, then we may decompose $R$ as the countable union of compact Lipschitz graphs
    \begin{equation*}
        R \hneq \bigcup_{m=1}^\infty f_m(A_m)
    \end{equation*}
    Such that there exists $x_m\in\rn$, $t_m>0$ with $f_m(A_m) \subset B_{t_m}(x_m)$. 
    \medskip
  
    \noindent {\bf (ii)}: Suppose we have a closed $F$, a set of finite perimeter $E$, a rectifiable set $R\subset F^c$ with a decomposition as in {\bf (i)}, and $\b\in(0,1)$. Then, there exists $M\in\N$ and open sets of finite perimeter $D_1,\dots,D_M$ such that 
    \begin{align}
        \h(R\setminus\bigcup_{m=1}^M \overline{D_m}) < \frac{1}{2}\h(R),\\\label{boundary of the replacement pieces}
        \overline{D_m} \cap \overline{D_{m'}} = \emptyset,\,\, {\rm for}\,\,1\le m<m'\le M,\\\label{ds are separated}
        d(F,D_m) > 0,\\
        f_m(A_m) \subset D_m \subset B_{t_m}(x_m),\,\,\forall m\leq M,\\\label{novack nice boundary}
        \pa D_m \text{ is Lipschitz, and } \h(\pa D_m\setminus\pas D_m) = 0,\\
        \sum_{m=1}^M|D_m| \le \beta,\\
        \big|B_s(z)\cap \bigcup_{m=1}^M D_m\big| \leq \beta|B_s(z)|, \forall z\in F, s>0,\\\label{disjoint boundary condition}
        \h(\pa D_m\cap\pas E) = 0,\\
        P(D_m) \leq 2\h(f_m(A_m)) + \frac{\beta}{2^m},\,\, \forall m\leq M.
    \end{align}

    \medskip
  
    \noindent {\it Step three}: Fixing a $\de_j$ we construct a corresponding cluster satisfying the conditions of step 1 for this $\de_j$. In this step, we set up the iteration seen in \cite{Novack_2024} with the single modification that 
    \begin{equation*}
        K = \bigcup_{i=1}^{N+1}\pas\E(i).
    \end{equation*}
    \noindent
    We set $F = \emptyset = F_{-1},\, E = \E(N+1),\,R_0 = K\setminus\pas E$ and $\beta_0 = \min\{\frac{1}{8},\frac{\de^*}{4}\}$, where $\de^*$ will be specified later. This yields $M_0\in\N$ and open sets of finite perimeter $D_1^0,\dots,D^0_{M_0}$ by step two. We collect these sets by setting $\cald_0 = \cup_{m=1}^{M_0} D^0_m$. We then iterate by applying step two to $F = F_{k-1} = \cup_{m=0}^{k-1}\overline{\cald_m}$, $R=R_k = K\setminus(\pas E\cup F_{k-1})$ and $\beta_k = \min\{\frac{1}{2^{k+3}},\frac{\de^*}{2^{k+2}}\}$. This produces $M_k\in\N$, families of sets $D^k_1,\dots,D^k_{M_k}$ and $\cald_k = \cup_{m=1}^{M_k}D^k_m$ satisfying
    \begin{align}\label{Remove half of K each time}
        \h(K\setminus(\pas E \cup F_k))\leq \frac{1}{2}\h(K\setminus(\pas E \cup F_{k-1})),\\\label{distance DF is positive novack}
        d(\cald_k,F_{k-1})>0,\\\label{exp dec size}
        |\cald_k|\leq \frac{\de^*}{2^{k+2}},\\
        |B_s(z)\cap\cald_k|\leq\frac{|B_s(z)|}{2^{k+3}}\,\, \forall z\in F_{k-1},s>0,\\
        \h(\pa \cald_k \cap \pas E) = 0,\\\label{perimeter of set D novack}
        P(\cald_k)\leq 2\h(K\cap\overline{\cald_k}\setminus\pas E) + \frac{\de^*}{2^{k+2}}.
    \end{align}
    We utilize these $F_k$'s to construct a cluster $\F_k$ as follows:
    \begin{align*}
        \F_k(N+1) &= \E(N+1)\Delta F_{k},\\
        \F_k(h) &= \E(h) \setminus F_k.
    \end{align*}
    Before moving on to the next step, we take note that the set $F = \cup_{k}\ov{\cald_k}$ satisfies, by the argument of \cite[Step five of the proof of Theorem 2.3]{Novack_2024} (which uses only the volume bounds \eqref{exp dec size}, the density estimate on $\cald_k$ at points of $F_{k-1}$, and the Lipschitz regularity of $\pa\cald_k$),
    \begin{align}\label{volume of replacement set novack}
        |F|&\leq \frac{\de^*}{2},\\\label{limit of bdry is bdry of limit Novack}
        \pas F &\hneq \cup_{k} \pas\cald_k.
    \end{align}
    
    \medskip
  
    \noindent {\it Step four}: We take the limit of the clusters $\F_k$ and verify it satisfies the conditions of step one. As $|F\Delta F_k| \to 0$, it follows immediately that $d(\F,\F_k) \to 0$ where
    \begin{align*}
        \F(N+1) &= \E(N+1)\Delta F,\\
        \F(h) &= \E(h) \setminus F.
    \end{align*}
    Now we show that $\F$ is a wet $N$-cluster. First, from \eqref{boundary of the replacement pieces}, \eqref{disjoint boundary condition} and \eqref{limit of bdry is bdry of limit Novack} we observe that
    \begin{equation}\label{symmetric diff of bdry is union novack}
        \pas\F(N+1) \hneq \pas\E(N+1) \cup\pas F.
    \end{equation}    Then, for the (non-exterior) air chambers we see
    \begin{equation}\label{red bdry of liquid novack}
        \pas\F(h) \hneq (\pas\E(h) \cap F^{(0)}) \cup (\pas F \cap \E(h)^{(1)}) \cup \{\nu_{F} = -\nu_{\E(h)}\}.
    \end{equation}
    Thus to show $\F$ is a wet $N$-cluster it suffices by \eqref{symmetric diff of bdry is union novack} and \eqref{red bdry of liquid novack} to demonstrate
    \begin{equation*}
        \pas\E(h) \cap F^{(0)} \hnsub \pas\E(N+1).
    \end{equation*}
    Indeed, by \eqref{Remove half of K each time} and continuity of measure we see that 
    \begin{equation}\label{F eats K in novack construction}
        K \hnsub F \cup\pas\E(N+1).
    \end{equation}
    Because the $\cald_k$ are open sets and \eqref{novack nice boundary} holds, we see that 
    \begin{equation}\label{decomposition with novack nice boundary}
        \h(F\cap F^{(0)}) = 0.
    \end{equation}
    The combination of \eqref{F eats K in novack construction} and \eqref{decomposition with novack nice boundary} will yield
    \begin{align*}
        \h(\pas\E(h) \cap F^{(0)}\setminus \pas\E(N+1)) &\leq \h(K \cap F^{(0)}\setminus \pas\E(N+1))\leq \h(F\cap F^{(0)}) = 0,
    \end{align*}
    thus demonstrating that $\F$ is indeed a wet $N$-cluster. We see by \eqref{volume of replacement set novack} that 
    \begin{equation*}
        d(\E,\F) = \sum_{h=1}^{N+1}|\E(h)\Delta\F(h)| \leq (N+1)\de^*.
    \end{equation*}
    By \eqref{ds are separated}, \eqref{distance DF is positive novack} \eqref{perimeter of set D novack}, \eqref{limit of bdry is bdry of limit Novack} and \eqref{symmetric diff of bdry is union novack} we observe that
    \begin{align*}
        P_{\rm wet}(\F) &= P(\F(N+1)) = P(\E(N+1)) + P(F)\\
        &= P(\E(N+1)) + \sum_{k=1}^\infty P(\cald_k)\\
        &\leq P(\E(N+1)) + \sum_{k=0}^\infty 2\h(K\setminus(\pas\E(N+1)\cup\ov{\cald_k}) + \frac{\de^*}{2^{k+2}}\\
        &\leq P(\E(N+1)) + 2\h(K\setminus\pas\E(N+1)) + \frac{\de^*}{2}\\
        &= P^*_{\rm wet}(\E) + \frac{\de^*}{2}.
    \end{align*}
    Choosing $\de^* \leq \frac{\de_j}{N+1}$ we observe $\F$ satisfies the conditions of step one of the theorem.
\end{proof}

\begin{proof}
    [Proof of Theorem \ref{relaxation thm main}] 
    Fix $j\in\N$ and let $\E$ be an $(N+1)$-cluster with $|\E| = (v,\e)$ and $\ps(\E) \leq \varphi_\e(v) + \frac1j$. By Theorem \ref{approximation thm main later} we may choose a wet $N$-cluster $\F_j$ with $|\F_j(h)| = |\E(h)|$ and $$\psi_\e(v) \leq P_{\rm wet}(\F_j) \leq \ps(\E) + \frac{1}{j} \leq \varphi_\e(v) + \frac{2}{j}.$$
    Letting $j\to \infty$ shows $\psi_\e(v) \leq \varphi_\e(v).$ Since for any wet $N$-cluster $\F$ it holds that $\ps(\F) = P_{\rm wet}(\F)$, we trivially have $\varphi_\e(v)\leq \psi_\e(v).$
\end{proof}

\begin{remark}
    By Lemma \ref{truncation for wet cluster}, we observe that a minimizer $\E$ for $\varphi_\e(v)$ (as in \eqref{relaxed problem defn}) must be bounded. Otherwise, we may truncate and apply a volume fixing variation to find a competitor with lower energy. Following the proof of Theorem \ref{approximation thm main later}, we may construct a uniformly bounded sequence of wet clusters $\E_j$ such that $|\ps(\E) - P_{\rm wet}(\E_j)| \to 0$ and $|\E_j(h)| = |\E(h)|$. Thus all minimizers to the relaxed problem remain limits of wet clusters, further validating our choice of relaxation.
\end{remark}

\section{Regularity of generalized minimizers}
In this section, we prove Theorem \ref{regularity thm main}. The crux of this proof is demonstrating that for a $\varphi_\e(v)$-minimal cluster, each dry chamber is a $(\La,r)-$minimizer. This allows us to deduce that the reduced boundary of each dry chamber is smooth. We then derive the Euler-Lagrange equations and show that each interface $\pas\E(i)\cap\pas\E(j)$ is constant mean curvature. Finally, we borrow regularity theory from the theory of free boundary problems to deduce a $C^{1,1}$ regularity for the transition region between chambers $\pa\E(N+1)\setminus\pas\E(N+1)$.

\begin{proof}
    [Proof of Theorem \ref{regularity thm main}] \noindent {\it Step one}: We show there exist positive constants $\Lambda$ and $r_0$ such that $\E(h)$ is a $(\Lambda,r_0)$-perimeter minimizer for $h = 1\dots N$. Without loss of any generality, we prove this for $\E(N)$. The proof for $h = 0$ follows via the exact same rather long computations, and is therefore omitted in the paper without much conceptual loss. To start, observe that by a standard volume fixing variations argument (such as the one in \cite{Maggi_Book} Theorem 29.14), we may find a $\Lambda'$ and $r_0$ such that
    \begin{equation}\label{vol fixing variations wet version}
        \ps(\E) \leq \ps(\F) + \Lambda' d(\F,\E),
    \end{equation}
    as long as diam$(\F\Delta\E) < 2r_0.$ Indeed, consider some set $V$ such that $V\Delta E \subset\subset B_{r_0}(x)$ for some $x$. Then we define a new cluster $\F$ as follows:
    \begin{align*}
        \F(N+1) &= (\E(N)\cup \E(N+1))\Delta V\\
        \F(N) &= (\E(N) \cup \E(N+1))\cap V\\
        \F(h) &= \E(h) \setminus V \text{ for } h = 1,2\dots N-1
    \end{align*} 
    The majority of step one will be dedicated to showing the inequality
    \begin{equation}\label{difference in wet energies replacement}
        \ps(\F) - \ps(\E) \leq P(V) - P(\E(N)).
    \end{equation}
    For brevity, we define the sets 
    $$K_\E = \bigcup_{h=1}^{N}\pas\E(h)\setminus \pas\E(N+1), \,\, K_\F = \pas V \cup \bigcup_{h=0}^{N-1}\pas\E(h).$$
    By the definition of $\ps$, we have that
    \begin{equation}\label{wet perimeter of competitor}
        \ps(\F) = \h(\pas\F(N+1)) + 2\h(K_\F\setminus \pas\F(N+1)).
    \end{equation}
    To compute the first term, we see by \cite{Maggi_Book} Theorem 16.3 that 
    \begin{align*}
        \h(\pas\F(N+1)) = \h(\pas(\E(N)\cup \E(N+1))\Delta \pas V)
    \end{align*}
    Defining the set $A = \pas(\E(N)\cup \E(N+1))\cap \pas V$, because $\{\nu_{\E(N)} = \nu_{\E(N+1)}\}$ is $\h$ null we have that 
    \begin{align*}
        \pas\F(N+1) &\hneq [\pas(\E(N)\cup \E(N+1))\cup \pas V]\setminus A\\ &\hneq [\pas V \cup (\pas \E(N) \cap \E(N+1)^{(0)}) \cup (\E(N)^{(0)} \cap \pas\E(N+1))]\setminus A
    \end{align*}
    Now by Federer's theorem and by the fact that $\{\E(h)\}_{h=0}^{N+1}$ is a Caccioppoli partition of $\R^{n+1}$, we may rewrite this as 
    \begin{align}\nonumber
        \pas\F(N+1) &\hneq [\pas V \cup (\pas\E(N) \cap \bigcup_{h=0}^{N-1} \pas \E(h)) \cup (\pas \E(N+1) \cap \bigcup_{h=0}^{N-1} \pas \E(h)]\setminus A\\\label{Boundary of replacement chamber wrt K}
        &= \Big[\pas V \cup \Big([\pas \E(N) \cup \pas\E(N+1)] \cap \bigcup_{h=0}^{N-1} \pas \E(h)\Big) \Big] \setminus A.
    \end{align}
    Thus we see that 
    \begin{align}\nonumber
        \h(\pas\F(N+1)) &= \h(\pas V) - \h(A) \nonumber\\
        &\quad + \h\big((\pas \E(N) \cup \pas \E(N+1)) \cap\bigcup_{h=0}^{N-1} \pas \E(h) \setminus \pas V\big)\\\label{measure of the wet chamber in replacement}
        &= \h\big(\bigcup_{h=0}^{N-1} \pas \E(h) \cap (\pas\E(N) \cup \pas \E(N+1))\big) + \h(\pas V) - 2\h(A).
    \end{align}
    Where we implicitly have used that 
    $$\pas V \cap \bigcup_{h=0}^{N-1} \pas \E(h) \cap (\pas\E(N) \cup \pas\E(N+1)) \hneq A.$$
    We next show that
    \begin{equation}\label{kf containment}
        \bigcup_{h=1}^{N+1} \pas\F(h) \hnsub \F(N+1) \cup K_\F.
    \end{equation}
    This is immediate for all $h\neq N$. For $h = N$, we apply \cite{Maggi_Book} Theorem 16.3 to see that
    \begin{align*}
        \pas\F(N) \hneq& [\pas V \cap (\E(N) \cup \E(N+1))^{(1)}] \cup[V^{(1)} \cap \pas(\E(N)\cup\E(N+1))] \\&\cup \{\nu_V = \nu_{\E(N)\cup\E(N+1)}\}.
    \end{align*}
    As $\pas V \subset K_\F$ it suffices to show
    \begin{align*}
        \pas(\E(N)\cup\E(N+1)) \hnsub K_\F,
    \end{align*}
    but this was done in \eqref{Boundary of replacement chamber wrt K}. Thus we have shown \eqref{kf containment}, which we may use to bound the multiplicity 2 energy of $\F$ by
    \begin{align}\label{K containment upper bound}
        \h(\bigcup_{h=0}^{N+1}\pas\F(h) \setminus \pas\F(N+1)) \leq  \h(K_\F \setminus\pas\F(N+1)).
    \end{align}
    By using \eqref{Boundary of replacement chamber wrt K} with the set $K_\F\setminus \pas\F(N+1),$ we observe 
    \begin{align}\nonumber
        K_\F\setminus \pas\F(N+1) &\hneq \big[\bigcup_{h=0}^{N-1}\pas\E(h) \setminus (\pas V \cup \pas \E(N) \cup \pas\E(N+1))\big] \cup (K_\F \cap A)\\\label{mult 2 replacement estimate}
        &\hnsub \big[\bigcup_{h=0}^{N-1}\pas\E(h) \setminus (\pas \E(N) \cup \pas\E(N+1))\big] \cup A.
    \end{align}
    By combining \eqref{K containment upper bound} and \eqref{mult 2 replacement estimate} we may show
    \begin{align*}
        \h(\bigcup_{h=0}^{N+1}\pas\F(h) \setminus \pas\F(N+1)) \leq \h\big(\bigcup_{h=0}^{N-1}\pas\E(h) \setminus (\pas \E(N) \cup \pas\E(N+1))\big) + \h(A).
    \end{align*}
    In tandem with \eqref{measure of the wet chamber in replacement}, this shows that
    \begin{equation*}
        \ps(\F) \le P(V) + \sum_{h = 0}^{N-1}P(\E(h)).
    \end{equation*}
    We recall \eqref{wet energy equivalent form} to see that the above immediately proves \eqref{difference in wet energies replacement}. Now by applying \eqref{vol fixing variations wet version} to \eqref{difference in wet energies replacement} and observing that $d(\E,\F) \leq (N+1)d(\E(N),V)$, we see that 
    \begin{align*}
        P(\E(N))\leq P(V) + (N+1)\Lambda' d(\E(N),V).
    \end{align*}
    Choosing $\Lambda = (N+1)\Lambda'$ we complete step one. Before proceeding, we remark that with \cite{Maggi_Book} Theorem 28.1 we immediately see that (i) holds.
    
    \medskip
  
    \noindent {\it Step two}: We compute the Euler-Lagrange equations for this problem to prove assertion (ii) of the theorem. In order to do so, we make the notational shortcut $\Sigma_{ij} = \pas\E(i) \cap\pas\E(j)$ and $\pa_{ij} = \pa\E(i)\cap\pa\E(j)$. We start by proving the following intermediate step:\\
    {\it Intermediate step}: For any $x\in \Sigma_{ij}\setminus\S'$ there exists a ball $B_x$ containing $x$ such that 
    \begin{equation}\label{intermediate step}
        \bigcup_{h=0}^{N+1} \pa\E(h) \cap B_x\subset\S_{ij}.
    \end{equation}
    We observe that $x\in \E(i)^{(\frac{1}{2})}\cap\E(j)^{(\frac{1}{2})}$, and therefore $\E(h)^{(0)}$ for $h\neq i,j$. As $\E(h)$ is a $(\Lambda,r_0)$ perimeter minimizer for $h \neq N+1$, it satisfies the density estimates of \cite{Maggi_Book} Theorem 21.11 implying that 
    \begin{equation}\label{boundary for lambda r clusters}
        x\notin \pa\E(h), \text{ for } h \neq i,j,N+1.
    \end{equation}
    If moreover $N+1\notin\{i,j\}$, then also $x\notin\pa\E(N+1)$: indeed $x\in\E(N+1)^{(0)}$ gives $x\notin\pas\E(N+1)$, and $x\notin\S'$ by assumption. Thus we may choose a ball $B'$ containing $x$ which is disjoint from the set
    \begin{align*}
        \bigcup_{h\neq i,j} \E(h) \cup\bigcup_{(h,k)\neq(i,j)}\pa_{hk}.
    \end{align*}
    Moreover, \cite{Maggi_Book} Theorem 26.5 implies we may take a smaller ball $B \subset B'$ such that $B \cap (\pa_{ij}\setminus\S_{ij}) = \emptyset$, where $B$ is clearly the desired ball. 
    
    From this, we may construct vector fields $Z_i\in C^1_c(\R^{n+1},\R^{n+1})$ for $i = 1,2\dots N+1$ which satisfy for an arbitrarily small parameter $\eta$
    \begin{equation}\label{fields for IFT for EL}
        \int_{\pas\E(j)} Z_i\cdot\nu_{\E(j)}\, d\h = \mathbf{1}\{i=j\} + O(\eta)\qquad\text{for } i,j = 1,\dots,N+1.
    \end{equation}
    Indeed, by \cite{Maggi_Book} Lemma 29.16, for every $i\in\{1,\dots,N+1\}$ there exist $N_i\in\N$ and distinct indices $h_0,h_1,\dots,h_{N_i}\in\{0,\dots,N+1\}$ such that
    \begin{equation*}
        h_0 = i,\qquad h_{N_i} = 0,\qquad \h(\S_{h_{j-1}h_j})>0\quad\text{for every } j = 1,\dots,N_i;
    \end{equation*}
    that is, every chamber is joined to the exterior by a chain of interfaces of positive area.
    By Federer's Theorem, for each $j = 1,\dots,N_i$ there exists $x_j\in\S_{h_{j-1}h_j}$ with $$x_j\in\E(h_{j-1})^{(\frac{1}{2})}\cap\E(h_{j})^{(\frac{1}{2})}.$$
    Then we may fix $\eta>0$ as a small parameter, and choose $r_j$ such that 
    \begin{equation}
        \frac{|B_{r_j}(x_j)\setminus (\E(h_{j-1})\cup\E(h_{j}))|}{|B_{r_j}(x_j)|}<\eta.
    \end{equation}
    Next, we choose $X_j \in C^1_c(\R^{n+1},\R^{n+1})$ such that supp$(X_j)\subset\subset B_{r_j}(x_j)$ as follows.
    We fix a cutoff $\phi\in C^\infty_c(B_1(0))$ such that 
    $\phi = 1 \text{ on } B_{\frac{1}{2}}(0) \text{ and } 0\leq\phi\leq1$.
    Letting $\nu_j = \nu_{\E(h_{j-1})}$, we define 
    \begin{align*}
        X_j &= c_j\,\phi\Big(\frac{x - x_j}{r_j}\Big)\nu_j(x_j) \eqqcolon c_j\,\phi_j\,\nu_j^*,\\
        c_j&\text{ chosen so that } \int_{\pas\E(h_{j-1})} X_j\cdot\nu_{j}\,d\h = 1 .
    \end{align*}
    We claim that $c_j \leq C(n)r_j^{-n}$ for $r_j$ sufficiently small ($r_j$ can be shrunk without loss of any generality). Indeed, 
    \begin{align}\label{normalization lower bound}
        \int_{\pas\E(h_{j-1})} \phi_j\nu_j^* \cdot \nu_j\,d\h &= \int_{\pas\E(h_{j-1})}\phi_j|\nu_j|^2 + \phi_j(\nu^*_j - \nu_j)\cdot\nu_j\,d\h\\\nonumber
        &\ge \h(\pas\E(h_{j-1})\cap B_{r_j/2}(x_j)) - \int_{\pas\E(h_{j-1})}\phi_j(\nu_j - \nu^*_j)\cdot\nu_j\,d\h .
    \end{align}
    Since $x_j\in\pas\E(h_{j-1})$, by \cite{Maggi_Book} Corollary 15.8, (15.9), we have
    $$\lim_{r\to0}\frac{\h(\pas\E(h_{j-1})\cap B_r(x_j))}{\om_n r^n} = 1,$$
    and thus, for $r_j$ small enough,
    $$\h(\pas\E(h_{j-1})\cap B_{r_j/2}(x_j)) \geq \frac{1}{2}\om_n\Big(\frac{r_j}{2}\Big)^n .$$
    By definition of the reduced boundary,
    \begin{align*}
        \frac{\int_{\pas\E(h_{j-1})}\phi_j(\nu_j - \nu^*_j)\cdot\nu_j\,d\h}{\h(\pas\E(h_{j-1})\cap B_{r_j}(x_j))} \leq \frac{\int_{\pas\E(h_{j-1})\cap B_{r_j}(x_j)}(\nu_j - \nu^*_j)\cdot\nu_j\,d\h}{\h(\pas\E(h_{j-1})\cap B_{r_j}(x_j))}\\
        = 1 - \nu_j^*\cdot\frac{\int_{\pas\E(h_{j-1})\cap B_{r_j}(x_j)} \nu_j\,d\h}{\h(\pas\E(h_{j-1})\cap B_{r_j}(x_j))} \to 0\quad\text{as } r_j\to0,
    \end{align*}
    where we used $(\nu_j-\nu_j^*)\cdot\nu_j = 1-\nu_j^*\cdot\nu_j\geq0$ and $0\leq\phi_j\leq1$ in the first inequality. This implies that for $r_j$ sufficiently small, $$\int_{\pas\E(h_{j-1})} \phi_j\nu_j^* \cdot \nu_j\,d\h \geq \frac{\om_n}{2}\Big(\frac{r_j}{2}\Big)^n - \frac{\om_n}{4}\Big(\frac{r_j}{2}\Big)^n = \frac{\om_n}{4}\Big(\frac{r_j}{2}\Big)^n.$$
    proving the claim. Therefore we see that $$|\Div X_j| \le \frac{c_j\|\nabla\phi\|_\infty}{r_j}\le\frac{C(n,\phi)}{r_j^{n+1}}.$$
    The divergence theorem and the fact that supp$(X_j)\subset B_{r_j}(x_j)$ will give for any $h$
    $$\Big|\int_{\pas\E(h)}X_j\cdot \nu_{\E(h)}\,d\h\Big| = \Big|\int_{\E(h)\cap B_{r_j}(x_j)}\Div X_j\,dy\Big| \leq \frac{C(n,\phi)}{r_j^{n+1}}|\E(h) \cap B_{r_j}(x_j)| .$$
    Thus for any $h\notin\{h_{j-1},h_j\}$, $$\Big|\int_{\pas\E(h)}X_j\cdot \nu_{\E(h)}\,d\h\Big| \leq C(n,\phi)\,\eta.$$
    Moreover, 
    $$\sum_{h=0}^{N+1}\int_{\pas\E(h)}X_j\cdot\nu_{\E(h)}\, d\h = \int_{\rn}\Div X_j\,dy = 0$$ implies that 
    $$\int_{\pas\E(h_j)}X_j\cdot\nu_{\E(h_j)}\, d\h = -1 + O(\eta).$$
    Finally, set $Z_i = \sum_{j=1}^{N_i}X_j$. Since the $h_j$ are distinct, the fluxes computed above telescope along the chain: the flux of $Z_i$ into $\E(h_0) = \E(i)$ is $1+O(\eta)$, into $\E(h_k)$ for $1\leq k<N_i$ it is $(-1+O(\eta))+1 = O(\eta)$, and into every chamber not on the chain it is $O(\eta)$; this is \eqref{fields for IFT for EL}.
    With this in mind, consider an arbitrary $X\in C^1_c(\R^{n+1},\R^{n+1})$, and define the functions
    \begin{align*}
        f_{t,s}(x) &= x + tX(x) + \sum_{i=1}^{N+1}s_iZ_i(x)\\
        \Phi(t,s) &= \big(|f_{t,s}(\E(1))|,\dots,|f_{t,s}(\E(N+1))|\big)\in\R^{N+1}
    \end{align*}
    By \eqref{fields for IFT for EL} and the first variation of volume, $D_s\Phi(0) = I + O(\eta)$, which is invertible for $\eta$ small. Thus we may apply the implicit function theorem to see that in some neighborhood of the origin, there exists $s = s(t)$ such that $\Phi(t,s(t)) = \Phi(0,0)$ For brevity, we define $g_t = f_{t,s(t)}$. Next, we may define the function 
    \begin{align*}
        m(t) = \sum_{h=0}^{N} P\big(g_t(\E(h))\big)
    \end{align*}
    By Theorem \ref{relaxation thm main}, we may conclude that $m'(0) = 0$. By the first variation of perimeter and the chain rule, we see that
    \begin{align*}
        m'(0) &= \sum_{h=0}^{N}\int_{\pas\E(h)} \Div^{\pas\E(h)}\Big(X + \sum_{i=1}^{N+1}s_i'(0)Z_i\Big)\,d\h,\\
        s'(0) &= -D_t\Phi(0),\qquad (D_t\Phi(0))_i = \int_{\pas\E(i)}X\cdot\nu_{\E(i)}\,d\h.
    \end{align*}
    Observe the vector $\lambda\in\R^{N+1}$ with components
    \begin{align*}
        \lambda_i = \sum_{h=0}^N\int_{\pas\E(h)}\Div^{\pas\E(h)}(Z_i)\, d\h
    \end{align*}
    is independent of our choice of $X$, and therefore we observe that 
    \begin{align}\label{EULER LAGRANGE}
        \sum_{h=0}^N\int_{\pas\E(h)} \Div^{\pas\E(h)}X\,d\h - \sum_{i=1}^{N+1} \lambda_i\int_{\pas\E(i)} X\cdot\nu_{\E(i)}\,d\h = 0.
    \end{align}
    Before proceeding, we observe that our intermediate step plus our Euler-Lagrange equations above will imply that $\S_{ij}$ is a constant mean curvature hypersurface, and thus smooth. This completes the proof of assertion (ii). 
    
    \medskip
  
    \noindent {\it Step three}: We prove the $C^{1,\a}$ version of (iii), that is, \eqref{c11 cap K}, \eqref{c11 bdry} and \eqref{c11 chamber} with $C^{1,\a}$ in place of $C^{1,1}$, and we express the curvatures of the interfaces in terms of the Lagrange multipliers of Step two. For $x\in \pa\E(N+1) \setminus(\pas\E(N+1) \cup \S)$, by \eqref{boundary for lambda r clusters} there exist $i,j$ such that $x\in \pa_{i,N+1} \cap\pa_{j,N+1}\cap \S_{ij}$ and $x\notin\pa_{hk}$ for $(h,k)\neq (i,N+1)$ or $(j,N+1)$. Therefore we may take a ball $B_r(x)$ such that 
    \begin{equation}\label{parition for tangential splitting}
        \{B_r(x) \cap \E(i), B_r(x)\cap\E(j), B_r(x)\cap\E(N+1)\} \text{ is an } \h \text{ partition of } B_r(x).
    \end{equation}
    Up to further shrinking $r$, by \cite{Maggi_Book} Theorem 26.3 we know that there exists a direction $\nu = \nu_{\E(j)}(x)$ and $u_1,u_2\in C^{1,\alpha}(\DD^\nu_r(x),(-\frac{r}{4}, \frac{r}{4}))$ with $u_1(x) = u_2(x) = 0$, $\nabla u_1(x) = \nabla u_2(x) = 0$ and
    \begin{align*}
        \E(i) \cap \CC^\nu_r(x) &= \{y + t\nu : y\in \DD^\nu_r(x), r > t > u_2(y)\},\\
        \E(j) \cap \CC^\nu_r(x) &= \{y + t\nu : y\in \DD^\nu_r(x), -r < t < u_1(y)\}.
    \end{align*}
    Along with \eqref{parition for tangential splitting}, this further implies that 
    \begin{align*}
        \E(N+1)\cap \CC^\nu_r(x) &= \{y+t\nu : y\in \DD^\nu_r(x), u_1(y) < t < u_2(y)\}.
    \end{align*}
   The above statements are equivalent to $C^{1,\alpha}$ analogues \eqref{c11 cap K}, \eqref{c11 bdry} and \eqref{c11 chamber}. Finally, we record that \eqref{EULER LAGRANGE} implies that the mean curvature of $\S_{ij}$ oriented by $\nu_{\E(i)}$ is given by 
    \begin{equation}\label{curvature in terms of lagrange multipliers}
        H_{ij} = \frac{\lambda_i - \lambda_j}{1 + \mathbf{1}\{i,j\neq N+1\}},
    \end{equation}
    where $H_{ij}$ denotes the scalar mean curvature of $\S_{ij}$ oriented by $\nu_{\E(i)}$, normalized by $\int_{\S_{ij}}\Div^{\S_{ij}}X\,d\h = \int_{\S_{ij}}H_{ij}\,X\cdot\nu_{\E(i)}\,d\h$ for every $X\in C^1_c$ supported in a ball as in the intermediate step; thus a sphere has positive mean curvature with respect to its outer normal, $H_{ji} = -H_{ij}$, and we have set $\lambda_0 = 0$.
    The sign information \eqref{averaged pressure inequality} is obtained in Step four.
    
    \medskip
  
    \noindent {\it Step four}: We complete the proofs of (iii) and (iv). By step three, for $x\in \S'\setminus\S$ we may find $\nu\in\SS^{n},r>0$ and $u_1,u_2\in C^{1,\a}(\DD^\nu_r(x))$ satisfying \eqref{c11 cap K}, \eqref{c11 bdry} and \eqref{c11 chamber}. If MS$(u)$ is the distributional mean curvature operator given by
    $$\text{MS}(u)[\varphi] = \int_{\DD^\nu_r(x)}\frac{\nabla u\cdot\nabla \varphi}{\sqrt{1+|\nabla u|^2}}d\h,$$
    then \eqref{curvature in terms of lagrange multipliers} will imply that ($\MS(u)$ being the mean curvature of the graph of $u$ oriented by $\nu$, which is $\nu_{\E(j)}$ along the graph of $u_1$ and $\nu_{\E(N+1)}$ along the graph of $u_2$)
    \begin{align}\label{lower cusp curvature}
        \text{MS}(u_1) &= H_{j,N+1} \text{ on } \{u_1<u_2\}\\\label{upper cusp curvature}
        \text{MS}(u_2) &= H_{N+1,i} \text{ on } \{u_1<u_2\}.
    \end{align}
    Then we define $u_+ = u_2 - u_1$ and claim that 
    \begin{equation}\label{Maximum principle for curvatures}
        \MS(u_2) - \MS(u_1) < 0\text{ on } \{u_1<u_2\}.
    \end{equation}
    We begin by observing that for all $\varphi\in C_c^\infty(\{u_1<u_2\})$,
    \begin{equation}\label{weak form curvature}
        \MS(u_2)[\varphi] - \MS(u_1)[\varphi] = \int_{u_1<u_2} A(x)[\nabla u_+]\cdot\nabla\varphi d\h = -\int_{u_1<u_2} \varphi \text{ div}(A\nabla u_+)d\h
    \end{equation}
    where 
    \begin{equation*}
        A(x) = \int_0^1\nabla^2F(s\nabla u_1(x) + (1-s)\nabla u_2(x))ds,
    \end{equation*}
    where $F(p) = \sqrt{1+|p|^2}$.
    Thus \eqref{lower cusp curvature}, \eqref{upper cusp curvature} and \eqref{weak form curvature} imply that 
    \begin{align*}
        \MS(u_2) - \MS(u_1) &= -\text{div}(A\nabla u_+) = H_{N+1,i} - H_{j,N+1} \eqqcolon C \text{ on } \{u_1<u_2\},\\
        u_+ &\geq 0 \text{ on } \{u_1<u_2\},\\
        u_+ &= 0 \text{ on } \{u_1=u_2\} \supset \partial\{u_1<u_2\}\cap\DD^\nu_r(x).
    \end{align*}
    By step three, $u_1,u_2\in C^{1,\a}(\DD^\nu_r(x))$, so that $A\in C^{0,\a}(\DD^\nu_r(x);\R^{n\times n}_{\mathrm{sym}})$ is uniformly elliptic. We claim that $C<0$. Indeed, $u_+\geq0$ vanishes together with $\nabla u_+$ on $\{u_1=u_2\}\ni x$. Pick $x_0\in\{u_1<u_2\}$ with $|x_0-x|<r/2$ and let $B_\de(x_0)$ be the largest ball centered at $x_0$ contained in $\{u_1<u_2\}$; since $\de\le|x_0-x|<r/2$ we have $\overline{B_\de(x_0)}\subset\DD^\nu_r(x)$, and $\pa B_\de(x_0)$ touches $\{u_1=u_2\}$ at some $y_0$. If $C\geq0$, then $u_+\in C^1(\overline{B_\de(x_0)})$ is a nonconstant ($u_+(x_0)>0=u_+(y_0)$) weak solution of $-\Div(A\nabla u_+)\geq0$ in $B_\de(x_0)$ attaining its minimum at $y_0\in\pa B_\de(x_0)$, so the boundary point lemma for divergence-form operators with Dini continuous coefficients \cite[Theorem 2.1]{ApushkinskayaNazarov2019} gives $\pa u_+/\pa n(y_0)<0$, contradicting $\nabla u_+(y_0)=0$. Thus we conclude that
    \begin{equation}\label{proper sign on divergence difference}
        \text{div}(A\nabla u_+) = H_{j,N+1} - H_{N+1,i} > 0.
    \end{equation}
    Equivalently $\lambda_i+\lambda_j>2\lambda_{N+1}$, which is \eqref{averaged pressure inequality} with $H_1 = \MS(u_1)$, $H_2 = \MS(u_2)$; the remaining assertions of (iii) are immediate consequences (for the exterior use $\lambda_0 = 0$).
    Thus we may apply the regularity theory of \cite{SilvTM}, in the form of \cite[Section 6, (6.1)]{SilvTM}: $u_2\geq u_1$ are $C^{1,\a}$, on $\{u_1<u_2\}$ they solve the mean curvature equation with the constant right-hand sides $\lambda_{N+1}-\lambda_i$ and $\lambda_j-\lambda_{N+1}$, whose difference is the positive constant $-C$, and on $\mathrm{int}\{u_1=u_2\}$ the common function solves it with the average $(\lambda_j-\lambda_i)/2 = H_{ji}$ by Step two, to show that \begin{equation}\label{c11 full result}
        u_1,u_2\in C^{1,1}(\DD^\nu_r(x)).
    \end{equation} Moreover, \eqref{proper sign on divergence difference} in conjunction with \eqref{c11 full result} (which makes $A$ Lipschitz) also allows us to apply the regularity theory for free boundaries developed in \cite{focardi2013monotonicityformulasobstacleproblems}. We conclude
    \begin{equation*}
        \text{FB} = \DD^\nu_r(x)\cap\{u_+ = 0\}
    \end{equation*}
    can be partitioned into sets Reg and Sing such that Reg is relatively open in FB and such that for every $z\in$ Reg there are $r>0$ and $\beta \in (0,1)$ such that $B_r(z) \cap \text{FB}$ is a $C^{1,\b}$ embedded $(n-1)$-dimensional manifold and such that Sing $= \cup_{k=0}^{n-1}\text{Sing}_k$ is relatively closed in FB with each $\text{Sing}_k$ locally $\H^k$-rectifiable in $\DD^\nu_r(x).$ As we have already shown \eqref{c11 chamber} holds, we observe that it implies 
    \begin{equation*}
        \CC^\nu_r(x) \cap\big(\pa\E(N+1)\setminus\pas\E(N+1)\big) = \{y + u_1(y)\nu:y\in\text{FB}\}.
    \end{equation*}
    Statement (iv) from the theorem follows immediately by a covering argument.
\end{proof}

\section{Convergence to the isoperimetric cluster problem}
This section is dedicated to the proof of Theorem \ref{convergence to plateau}. 
\begin{proof}
    [Proof of Theorem \ref{convergence to plateau}] 
  
    \medskip
  
    \noindent {\it Step one}: We show that $\varphi_0$ is continuous on $\R^N_+ = \{v\in\R^N: v(h) > 0\,\,\forall h\}$. For an $N$-cluster $\F$ we write $\Phi(\F) = \sum_{h=0}^N P(\F(h)) = 2P(\F)$, so that $\varphi_0(v) = \inf\{\Phi(\F):|\F| = v\}$ by \eqref{dry problem defn}. We claim that there is $C = C(n)$ such that for all $v,w\in\R^N_+$,
    \begin{equation}\label{scaling comparison}
        \varphi_0(w) \leq \lambda^n\varphi_0(v) + C\sum_{h=1}^N\big(w(h) - \lambda^{n+1}v(h)\big)^{\frac{n}{n+1}}, \qquad \lambda = \min_{1\leq h\leq N}\Big(\frac{w(h)}{v(h)}\Big)^{\frac{1}{n+1}}.
    \end{equation}
    Indeed, by the choice of $\lambda$ we have $\lambda^{n+1}v(h)\leq w(h)$ for every $h$, so every term in the sum is nonnegative. Let $\F$ be a minimizer for $\varphi_0(v)$, which exists and is bounded by \cite{Maggi_Book} Theorem 29.1. The scaled cluster $\lambda\F$ satisfies $|\lambda\F(h)| = \lambda^{n+1}v(h)\leq w(h)$ and $\Phi(\lambda\F) = \lambda^n\Phi(\F)$. For each $h$ with $s_h = w(h)-\lambda^{n+1}v(h)>0$, let $B_h$ be a ball of volume $s_h$, with the $B_h$ pairwise disjoint and at positive distance from $\bigcup_{k\geq1}\lambda\F(k)$, and set $\F'(h) = \lambda\F(h)\cup B_h$ (with $\F'(h) = \lambda\F(h)$ if $s_h = 0$). Then $|\F'| = w$, and since $\pa B_h$ is counted once in $P(\F'(h))$ and once in $P(\F'(0))$,
    \begin{equation*}
        \Phi(\F') = \lambda^n\Phi(\F) + 2\sum_{h=1}^N P(B_h) = \lambda^n\varphi_0(v) + 2(n+1)\omega_{n+1}^{\frac{1}{n+1}}\sum_{h=1}^N s_h^{\frac{n}{n+1}}.
    \end{equation*}
    As $\F'$ is a competitor for $\varphi_0(w)$, this proves \eqref{scaling comparison}.

    Now let $v_j\to v$ in $\R^N_+$. Applying \eqref{scaling comparison} with $w = v_j$ gives $\limsup_j\varphi_0(v_j)\leq\varphi_0(v)$, since $\lambda\to1$ and $v_j(h)-\lambda^{n+1}v(h)\to0$. Applying it with the roles of $v$ and $v_j$ exchanged, that is with $\lambda_j = \min_h(v(h)/v_j(h))^{1/(n+1)}\to1$, gives
    \begin{equation*}
        \varphi_0(v)\leq\lambda_j^n\varphi_0(v_j) + C\sum_{h=1}^N\big(v(h)-\lambda_j^{n+1}v_j(h)\big)^{\frac{n}{n+1}},
    \end{equation*}
    and hence $\varphi_0(v)\leq\liminf_j\varphi_0(v_j)$. Thus $\varphi_0$ is continuous on $\R^N_+$; in particular it is lower semicontinuous, which is all that is used below.

    \medskip
  
    \noindent {\it Step two}: We can now show convergence of energies 
    \begin{equation}\label{convergence of energies}
        \lim_{\e\downarrow0}\varphi_\e(v) = \varphi_0(v).
    \end{equation}
    Given a sequence $\e_k \to 0$ and minimizers $\E_k$ of $\varphi_{\e_k}(v)$ (which exist by Theorems \ref{existence thm main} and \ref{relaxation thm main}), choose a sequence $\eta_j \downarrow 0$. By Lemmas \ref{truncation for wet cluster} and \ref{nucleation lemma}, we may choose a modified sequence $\E^j_k$ and parameters $\{\de^j_k(1),\dots,\de^j_k(N)\}$ such that
    \begin{align}\label{truncation decreases perimeter}
        \ps(\E^j_k) &\leq \ps(\E_k),\\\label{volume control on modified}
        |\E_k(h)| - \de^j_k(h) &\leq|\E^j_k(h)| \leq |\E_k(h)|,\\\label{uniform bound on truncation losses}
        \de^j_k(h) &\leq \eta_j,
    \end{align}
    and $\E^j_k$ is a uniformly bounded sequence of clusters in $k$. Here the nucleation lemma is applied to the air chambers only (the liquid chamber satisfies $|\E_k(N+1)\setminus F|\le\e_k$ trivially for $k$ large), and uniform boundedness is achieved, as in the proof of \cite{Maggi_Book} Theorem 29.1, by translating the finitely many pieces of the truncated cluster lying in far-apart balls to a common bounded region, which changes neither volumes nor energy.
    Thus the sequence $\E^j_k$ is compact so we may take a limiting $N$-cluster such that up to subsequences $\E^j_k \to \E^j$. By \eqref{volume control on modified} and \eqref{uniform bound on truncation losses}, there exists a $\de^j\leq \eta_j$ such that $|\E^j_k(h)|\to v(h) -\de^j(h).$ By lower semicontinuity of perimeter, and \eqref{truncation decreases perimeter}, we see that
    \begin{equation*}
        \liminf_{k\to\infty}\ps(\E_k)\geq \Phi(\E^j) \geq \varphi_0(v - \de^j).
    \end{equation*}
    As this holds for all $j$, we conclude by step one
    \begin{equation}\label{pre lsc bound on energies}
        \liminf_{k\to\infty}\ps(\E_k) \geq \liminf_{j\to\infty}\varphi_0(v - \de^j) \geq \varphi_0(v).
    \end{equation}
    Next, let $\F'$ be a minimal $N$-cluster for the energy $\varphi_0(v)$ and we can define an $(N+1)$-cluster $\F$ by defining the wet chamber to be a ball of volume $\e$ disjoint from the chambers of $\F'$. Then there is some constant $C(n)$
    \begin{equation*}
        \ps(\F) = \varphi_0(v) + C(n)\e^\frac{n}{n+1}.
    \end{equation*}
    In other words,
    \begin{equation}\label{trivial upper bound}
        \varphi_\e(v)\leq\varphi_0(v) + C(n)\e^\frac{n}{n+1}.
    \end{equation}
    The combination of \eqref{pre lsc bound on energies} and \eqref{trivial upper bound} imply \eqref{convergence of energies}.

\end{proof}

\bibliographystyle{alpha}
\bibliography{bibliography}

\end{document}